%% file: Korinman_WeilRepresentations_ArXiv.tex
\documentclass[a4paper]{article}

\usepackage[small,nohug,heads=vee]{diagrams}
\usepackage{amsmath,amssymb,amsthm, amscd}
\usepackage{array}
\usepackage{dsfont}
\usepackage{epsfig}
\usepackage[a4paper]{geometry}

\input{macrosmath}

\begin{document}

\theoremstyle{plain}
\newtheorem{theorem}{Theorem}[section]
\newtheorem{main_theorem}[theorem]{Main Theorem}
\newtheorem{proposition}[theorem]{Proposition}
\newtheorem{corollary}[theorem]{Corollary}
\newtheorem{corollaire}[theorem]{Corollaire}
\newtheorem{lemma}[theorem]{Lemma}
\theoremstyle{definition}
\newtheorem{definition}[theorem]{Definition}
\newtheorem{Theorem-Definition}[theorem]{Theorem-Definition}
\theoremstyle{remark}
\newtheorem*{remark}{Remark}
\newtheorem{example}{Example}

\sloppy

\title{Irreducible factors of Weil representations and TQFT}

\author{Julien \textsc{Korinman}
\\ \small Universidade Federal de S\~ao Carlos
\\ \small Rodovia Washington Lu\'is, Km 235, s/n 
\\ \small S\~ao Carlos - SP, 13565-905
\\ \small email: \texttt{ julienkorinman@dm.ufscar.br}
}
\date{}
\maketitle


\begin{abstract} 
We give the decomposition into irreducible factors of \weil representations of $\ssp$ at even levels, generalizing the decompositions in \cite{Kl,CNS} at odd levels. We then derive the decomposition of the  quantum representations of $SL_2(\mathbb{Z})$ arising in the $SU(2)$ and $SO(3)$ TQFTs. As application we show that, when the level indexing the TQFT is not a multiple of $4$, the universal construction of \cite{BHMV2} applied to a cobordism category without framed links leads to the same TQFT.
\vspace{2mm}
\end{abstract}

\section{Introduction and statements}

\vspace{2mm}
\subsection{A brief history}
\vspace{2mm}
\par In this paper we study a family of unitary representations of the symplectic groups $\ssp$, indexed by some integer $p\geq 2$, which are related to number theory, mathematical physics and topology (see the next section for definitions). 
They first appeared in the work of Kloosterman in $1946$ (see \cite{Kl}) where they arise as modular transformations of spaces of theta functions. They were rediscovered independently by the physicist Shale (\cite{Sh}) following Segal (\cite{Se}) in $1962$ when the authors studied the Weyl quantization of the symplectic torus. Their construction has been generalized to arbitrary locally compact abelian groups by Weil in $1964$ (see \cite{Weil}). The ones we consider in this paper are associated to $\zn$. They also appeared independently in the work of Igusa (\cite{Ig}) and Shimura (\cite{Shi}) on theta functions. See also  \cite{LV} for another construction.

\vspace{2mm}
\par The mathematical physics community studied the semi-classical properties of the \weil representations associated to finite cyclic groups when the level $p$ tends to infinity as a model for quantum chaotical behavior (see \cite{HB,KR,FNdB,BDB}).
\vspace{2mm}

\par Topologists began to study these representations because they fit into the framework of  Topological Quantum Field Theories.
Their definition for even levels and arbitrary genus first appeared in \cite{Fu2, Go} in relation with $3$-manifold invariants which were studied in \cite{MOO} and further explored in \cite{DG} in the more general context of abelian invariants.
\par  The main motivation of the author for this paper was to obtain information on the Witten-Reshetikhin-Turaev representations of the mapping class groups, as defined in \cite{RT}, using a relation between the two families of representations in the genus one case.
\par The construction we will use in this paper is related to knot and skein theory following the topological point of view of \cite{ GU}. Though less standard that the number theoretical or geometrical construction, this point of view is more elementary, crucial in the proofs of Theorem \ref{irred_factors}  and makes more transparent the relation with the Witten-Reshetikhin-Turaev representations made in the last section.

\vspace{4mm}
 \subsection{Statements}
\vspace{2mm}
\par Given two integers $p\geq 2$ and $g\geq 1$, the \weil representations are projective unitary representations of the symplectic group $\ssp$
$$ \pi_{p,g} : \ssp\rightarrow \PGL(U_p^{\otimes g})$$
where $U_p$ is a free module of rank $p$ over the ring:
$$ \textbf{k}_p := \left\{ 
\begin{tabular}{ll}
 $\mathbb{Z} \left[ A,\frac{1}{2p} \right] /(\phi_{p}(A))$  , &when $p$ is odd.
\\ $\mathbb{Z} \left[ A,\frac{1}{p} \right] /(\phi_{2p}(A))$,  &when $p$ is even.
\end{tabular} \right.
$$
where $\phi_p\in \mathbb{Z}[X]$ represents the cyclotomic polynomial of degree $p$. 
\vspace{3mm}
\par In \cite{Kl}, Kloosterman gave a complete decomposition of the Weil representations when $g=1$ and $p$ is odd. His result was further generalized by Cliff, Mc Neilly and Szechtman in \cite{CNS} to arbitrary genus still at odd levels (see also \cite{Prasad}).
\par The main result of this paper is the extension of these decompositions to even levels.
\vspace{2mm}
\par Let $a, b\geq 2$   be two coprime non negative integers with $b$ odd, and let $u$ and $v$ be odd integers such that $au+bv=1$ in the case where $a$  is odd  and such that $2au+bv=1$ if $a$ is even and $b$ is odd. We define a ring isomorphism $\mu : \textbf{k}_{ab}\rightarrow \textbf{k}_a\otimes \textbf{k}_b$ by $\mu(A)=(A^{vb},A^{au})$ if $a$ is odd and $\mu(A)=(A^{vb},A^{2au})$ if $a$ is even,   which turns $U_a^{\otimes g}\otimes U_b^{\otimes g}$ into a $\textbf{k}_{ab}$-module.
\vspace{1mm}
\par For $r$ prime and $n\geq 1$, we define the ring homomorphism  $\mu : \textbf{k}_{r^{n}}\rightarrow \textbf{k}_{r^{n+2}}$ by $\mu(A)=A^{r^2}$ which turns $U_{r^{n}}^{\otimes g}$ into a $\textbf{k}_{r^{n+2}}$-module.
\vspace{1mm}

\par Set $\sigma(p)$ for the number of divisors of $p$ including $1$.
\vspace{2mm}
\begin{theorem}\label{main_theorem}
The level $p$ \weil representation  contains $\sigma(p)$ irreducible submodules, when $p$ is odd and $\sigma(\frac{p}{2})$, when $p$ is even. They decompose according to the following rules, where $\cong$ denotes an isomorphism of $\ssp$ projective modules:
\begin{enumerate}
\item  If $a, b\geq 2$  are two coprime integers, then:
$$U_a^{\otimes g}\otimes U_b^{\otimes g} \cong U_{ab}^{\otimes g}$$
\item  If $r$ is prime and $n\geq 1$, then:
$$U_{r^{n+2}}^{\otimes g}\cong U_{r^n}^{\otimes g}\oplus W_{r^{n+2}}^{\otimes g}$$
where  $W_{r^{n+2}}$ is a free submodule of $U_{r^{n+2}}$.
\item If $r$ is an odd prime, then:
$$U_{r^2}^{\otimes g}\cong \mathds{1}\oplus W_{r^2}^{\otimes g}$$
where $\mathds{1}$ denotes the trivial representation.
\item  Every factor $U_p^{\otimes g}, W_{r^n}^{\otimes g}$ for $p\geq 3$ decomposes into two invariant submodules, 
\begin{eqnarray*}
U_p^{\otimes g}\cong U_p^{g,+}\oplus U_p^{g,-} \\ W_{r^n}^{\otimes g}\cong W_{r^n}^{g ,+}\oplus W_{r^n}^{g,-}
\end{eqnarray*}
We call $U_p^{g,+}$ and $W_{r^n}^{g,+}$ the even modules and $U_p^{g,-}, W_{r^n}^{g,-}$ the odd modules.

\item  The application of the previous four rules decomposes any $U_p^{\otimes g}$ into a direct sum of modules of the form $B_{r_1}\otimes \ldots \otimes B_{r_k}$ with $r_1,\ldots,r_k$ distinct prime numbers and $B_{r_i}\in \{ U_{r_i}^{g, \pm }, \;W_{r_i^n}^{g, \pm } \}$. These modules are all irreducible and pairwise distinct.
\end{enumerate} 
\end{theorem}

\vspace{3mm}
\par The Witten-Reshetikhin-Turaev representations $V_p$ of $SL_2(\mathbb{Z})$ defined in \cite{RT} are projectively isomorphic to the odd submodule $U_p^-$ of the Weil representations (see \cite{FK} when $p$ is even, \cite{LW} when $p\equiv 1 \pmod{4}$ and the last section of this paper for a general proof). We deduce the following:

\vspace{2mm}
\begin{corollary}\label{coro}
We have the following decomposition into irreducible modules of the genus one $SO(3)$ and $SU(2)$ quantum representations at level $p$  of $SL_2(\mathbb{Z})$:

\begin{eqnarray*}
V_p &\cong& \bigoplus_{B\in E , B_1\in E_1, \ldots , B_k\in E_k} B\otimes B_1 \otimes \ldots \otimes B_k, \rm{ ~ when  ~} p \rm{ ~ is  ~even;}
\\ V_p &\cong& \bigoplus_{ B_1\in E_1, \ldots , B_k\in E_k} B_1 \otimes \ldots \otimes B_k ,\rm{ ~when ~ } p \rm{ ~ is  ~odd.}
\end{eqnarray*}
where $p=2^m r_1^{n_1}\ldots r_k^{n_k}$ is the factorization into primes and:

\begin{itemize}
\item If  $j$ is such that $n_j$ is odd, $E_j = \left\{ W_{r_j^{n_j-2a_j}}^{+},W_{r_j^{n_k-2a_j}}^{-} , U_{r_j}^{+}, U_{r_j}^{-} \; | \; 0\leq a_j \leq \left\lceil \frac{n_k}{2} \right\rfloor -1 \right\}$.
\item If  $j$ is such that $n_j$ is even, $E_j =  \left\{ W_{r_j^{n_j-2a_j}}^{+},W_{r_j^{n_k-2a_j}}^{-} , \mathds{1} \; | \; 0\leq a_j \leq \left\lceil \frac{n_k}{2} \right\rfloor -1 \right\}$.
\item If $m$ is odd, $E = \left\{ W_{2^{m-2a}}^{+}, W_{2^{m-2a}}^{-}  , U_2 \; | \; 0\leq a \leq \left\lceil \frac{m}{2} \right\rfloor -1 \right\}$.
\item If $m$ is even, $E = \left\{ W_{2^{m-2a}}^{+}, W_{2^{m-2a}}^{-}  , U_4^{+}, U_4^{-} \; | \; 0\leq a \leq \left\lceil \frac{m}{2} \right\rfloor -1 \right\}$.
\end{itemize}
with the condition that each summand $B\otimes B_1 \otimes \ldots \otimes B_k$ or $B_1 \otimes \ldots \otimes B_k$ contains an odd number of modules $U_p^{-}, W_{r^n}^{-}$.
\end{corollary}
\vspace{3mm}

\par \begin{example}
The \weil representation $(\pi_{500},U_{500})$ at level $500$ decomposes as follows:
\begin{eqnarray*}
U_{500}&\cong& U_{4}\otimes U_{125} \cong U_4\otimes(U_5\oplus W_{125})
\\ &\cong&  (U^+_4\otimes U^+_5)\oplus(U^-_4\otimes U^+_5)\oplus (U_4^+\otimes U_5^-)\oplus(U^-_4\otimes U^-_5) 
\\ &&\oplus (U_4^+\otimes W^+_{125})\oplus (U_4^+\otimes W_{125}^-)
\oplus (U_3^-\otimes W^+_{125})\oplus (U_3^-\otimes W_{125}^-)
\end{eqnarray*}
 In particular, we derive the following decomposition of the $SU(2)$-quantum representation $(\rho_{500},V_{500})$ in genus one at level $500$:
$$
 V_{500} \cong U_{500}^-\cong (U^-_4\otimes U^+_5)\oplus (U_4^+\otimes U_5^-) \oplus (U_4^-\otimes W^+_{125})\oplus (U_4^+\otimes W_{125}^-)
$$
where each factor in parenthesis is an irreducible factor.
\end{example}

\vspace{2mm}
\par The previous decomposition has the following application. The TQFTs defined in \cite{BHMV2} associate to each closed oriented surface $\Sigma$, a vector space $V_p(\Sigma)$. To a triple $(M, \phi, L)$, where $M$ is a closed oriented $3$ manifold, $\phi:\partial M \rightarrow \Sigma$ an orientation-preserving homeomorphism and $L\subset M$ an embedded framed link (possibly empty), the TQFT associates a vector $Z_p(M,\phi,L)\in V_p(\Sigma)$. Such vectors generate $V_p(\Sigma)$ by definition (see the last section). The following theorem was proved by Roberts  in the particular case where $p$ is prime  (it results from Lemma $2$ in \cite{Ro}).
\vspace{2mm}
\begin{theorem}\label{th_BHMV} If $4$ does not divides $p$, then in the $SU(2)$ and $SO(3)$ TQFTs (see \cite{BHMV2} for definitions), 
the vectors $Z_p(M,\phi, \emptyset)$, associated to cobordisms without framed links, generate $V_p(\Sigma)$.
\end{theorem}
\par It results that the universal construction of \cite{BHMV2} applied to a cobordism category without framed links leads to the same TQFTs. In particular these TQFTs are determined by their $3$ manifolds invariants without framed links (see the last section for details). It contrasts with the usual constructions (see \cite{BHMV2,Roberts94})  where standard generating sets for $V_p(\Sigma)$ are constructed from the skein modules of Handlebodies.

\vspace{2mm}
 $\textbf{Acknowledgements:}$ The author is thankful to R.Bacher,  C.Blanchet, F.Costantino, L.Funar, and J. March\'e for valuable discussions  and to G.Masbaum for pointed him the reference \cite{Prasad}. The author  acknowledges  support from the grant ANR $2011$ BS $01 020 01$ ModGroup, the GDR Tresses, the GDR Platon and the GEAR Network.
\vspace{5mm}

\section{Definition of the projective \weil representations}

The following section closely  follows the definitions from \cite{GU}.

\subsection{Heisenberg groups and Schr\"odinger representations}

\vspace{2mm}
\begin{definition} 
\begin{enumerate}
\item
  Let $p\geq 2$ and $M$ be a compact oriented $3$-manifold possibly with boundary. The  \textit{reduced abelian skein module} $\widetilde{\mathcal{T}}_p(M)$  is the $\textbf{k}_p$-module  generated by the isotopy classes of oriented banded links of ribbons in $M$ quotiented by the relations given by the abelian skein relations  of Figure $\ref{relations_skein}$ and by the submodule generated by the links made of $p$ parallel copies of the same ribbon.
\begin{figure}[!h] 
\centerline{\includegraphics[width=10cm]{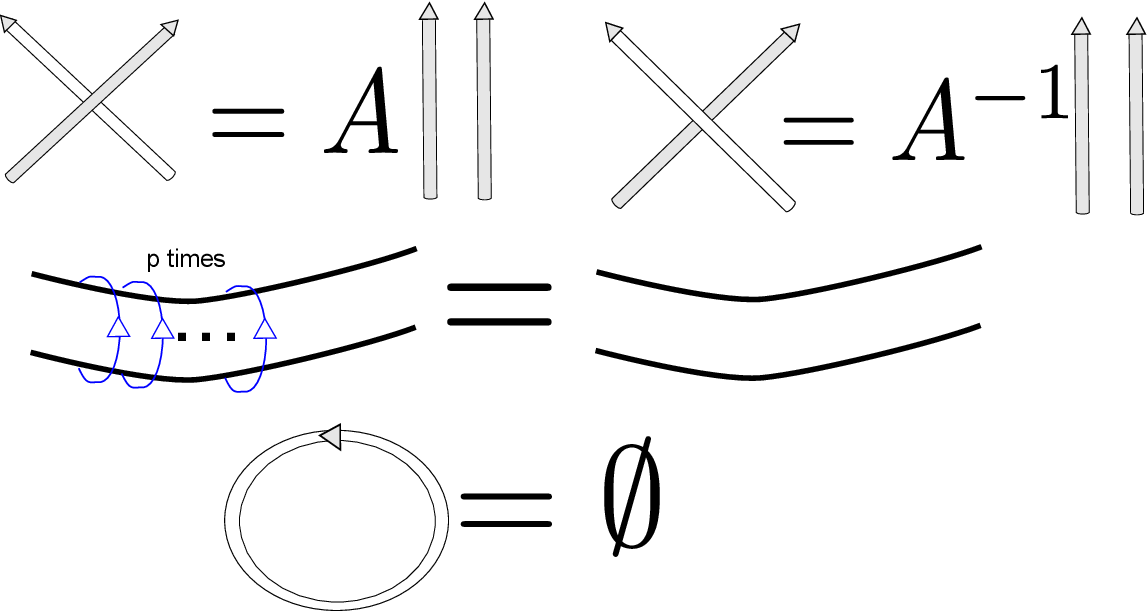}}
\caption{Skein relations defining the  reduced abelian skein modules.} 
\label{relations_skein} 
\end{figure} 

\vspace{2mm} 
\par The  reduced abelian skein module of the sphere $S^3$ has rank one. The class of a link $L\subset S^3$ in this module is equal to the class of the empty link multiplied by $A^{lk(L)}$ where $lk(L)$ represents the self-linking number of $L$. This gives a natural isomorphism $\widetilde{\mathcal{T}}_p(S^3)\cong \textbf{k}_p$. 
\vspace{2mm}
\par It is classic, that if $M\cong \Sigma \times [0,1]$ is a thickened surface, then its reduced skein module $\widetilde{\mathcal{T}}_p(M)$ is isomorphic to $\textbf{k}_p [H_1(\Sigma, \zn)]$.
\item
Denote by $H_g$  the genus $g$ handlebody. Its abelian skein module is freely generated by the elements of $H_1(H_g,\zn)$. So, if we denote by $U_p$ the module $\widetilde{\mathcal{T}}_p(S^1\times D^2)$, we have a natural $\textbf{k}_p$-isomorphism between $\widetilde{\mathcal{T}}_p(H_g)$ and $U_p^{\otimes g}$.
\item Let  $\Sigma_g$ be a closed oriented surface of genus $g$.
 The module $\widetilde{\mathcal{T}}_p(\Sigma_g\times [0,1])$ has an algebra structure with product induced by superposition, which appears to be the algebra of the following group. 
\vspace{2mm} \par  We denote by $c\in \widetilde{\mathcal{T}}_p(\Sigma_g\times [0,1])$ the product of the class of the empty link by $A\in \textbf{k}_p$. We call \textit{Heisenberg group}  and denote $\mathcal{H}_{p,g}$ the subgroup of $\widetilde{\mathcal{T}}_p(\Sigma_g\times [0,1])$ generated by $c$ and $H_1(\Sigma_g, \zn)$. Denote by $\omega$ the intersection form  $\omega : H_1(\Sigma_g, \zn)\times H_1(\Sigma_g, \zn) \rightarrow \zn$ when $p$ is odd and $\omega : H_1(\Sigma_g, \zn)\times H_1(\Sigma_g, \zn) \rightarrow \znn$, when $p$ is even. Then $\mathcal{H}_{p,g}$ is isomorphic to the group $H_1(\Sigma_g,\zn)\times \zn$, when $p$ is odd and $H_1(\Sigma_g, \znn)\times \znn$ when $p$ is even with group law given by: $$ (X,z)\bullet (X',z') = (X+X', z+z'+\omega(X,X'))$$
\item We choose a homeomorphism $\phi: \Sigma_g\rightarrow \Sigma_g$ so that  $\left( \Sigma_g\times [0,1] \right)  \bigcup_{\phi} H_g \cong H_g$. This gluing induces a linear action of the Heisenberg group on the reduced skein module $\widetilde{\mathcal{T}}_p(H_g)\cong U_p^{\otimes g}$. This representation is called \textit{the Schr\"odinger representation} and will be denoted by $\Add_p : \mathcal{H}_{p,g} \rightarrow \GL(U_p^{\otimes g})$. Up to isomorphism, this representation does not depend on $\phi$.
\end{enumerate}
\end{definition}

\subsection{The \weil representations}
 Every element of the mapping class group $\Mod (\Sigma_g)$ acts on $H_1(\Sigma_g,\mathbb{Z})$ by preserving the intersection form. Choosing a basis of $H_1(\Sigma_g, \mathbb{Z})$ we obtain  a surjective morphism $f : \Mod(\Sigma_g)\rightarrow \ssp$ whose kernel is called Torelli group.   
\vspace{2mm}
\par Let $g\geq 1$, the module $ \widetilde{\mathcal{T}}_p(\Sigma_g\times [0,1])$ is spanned by classes of links embedded in $\Sigma\times \{\frac{1}{2}\}$ with parallel framing whose class only depends on their homology class in $\Sigma_g$. The action in homology of the  mapping class group $\Mod (\Sigma_g)$ induces,  by passing through the quotient by the reduced skein relations, an action on the Heisenberg group. We denote by $\bullet$ this action. Let $\phi \in \Mod (\Sigma_g)$ and consider the representation $s^{\phi} : \mathcal{H}_{p,g}\rightarrow \GL(U_p^{\otimes g})$ defined by $s^{\phi}(h):= \Add_p(\phi \bullet h)$ for all $ h \in \mathcal{H}_{p,g}$. It is a standard fact, referred as the Stone-Von Neumann theorem, that 
the Schr\"odinger representation is the unique irreducible representation of the Heisenberg group sending the central element $c$ to the scalar operator $A\cdot \mathds{1}$.

It results that the representation $s^{\phi}$ is conjugate to the Schr\"odinger representation. Thus there exists $\pi_{p,g}(\phi) \in \GL(U_p^{\otimes g})$, uniquely determined up to multiplication by an invertible scalar, so that:

\begin{equation}\label{egorov}
\pi_{p,g}(\phi)\Add_p(h)\pi_{p,g}(\phi)^{-1} = \Add_p(\phi \bullet h),  \mbox{ for any } h\in \mathcal{H}_{p,g}
\end{equation}
\vspace{1mm}
The equation $(\ref{egorov})$ is called the Egorov identity and we easily show that the elements $\pi_{p,g}(\phi)$ define a projective representation $\underline{\pi}_{p,g} : \Mod(\Sigma_g)\rightarrow \PGL(U_p^{\otimes g})$ called the Weil representation.

\vspace{2mm} \par
 Since the action of $\Mod(\Sigma_g)$ on $\mathcal{H}_{p,g}$ factorizes through the Torelli group  and through $\spp$ when $p$ is odd and $\sppp$ when $p$ is even, so do the \weil representations. 

\vspace{2mm} \par The previous definition of the \weil representations as intertwining operators is not explicit. To manipulate it more easily, we choose the generators of $\ssp$ consisting of the image through $f$ of the Dehn twists $X_i, Y_i, Z_{ij}$ of Figure $\ref{gen_twists}$ (see  \cite{Li1} for a proof these Dehn twists generate the mapping class group). We define the basis $\{ e_{a_1}\otimes \ldots \otimes e_{a_g} | a_1,\ldots,a_g \in \zn \}$ of $U_p^{\otimes g}$ as in Figure $\ref{base_tore_solid}$, that means that  $e_{a_1}\otimes \ldots \otimes e_{a_g}$ is the class of a link made of $a_i$ parallel copies of an unframed ribbon encircling the $i^{th}$ hole of $H_g$ one time. To express the image of the generators in the basis, we will first need to define Gauss sums.

\begin{figure}[!h] 
\centerline{\includegraphics[width=11cm]{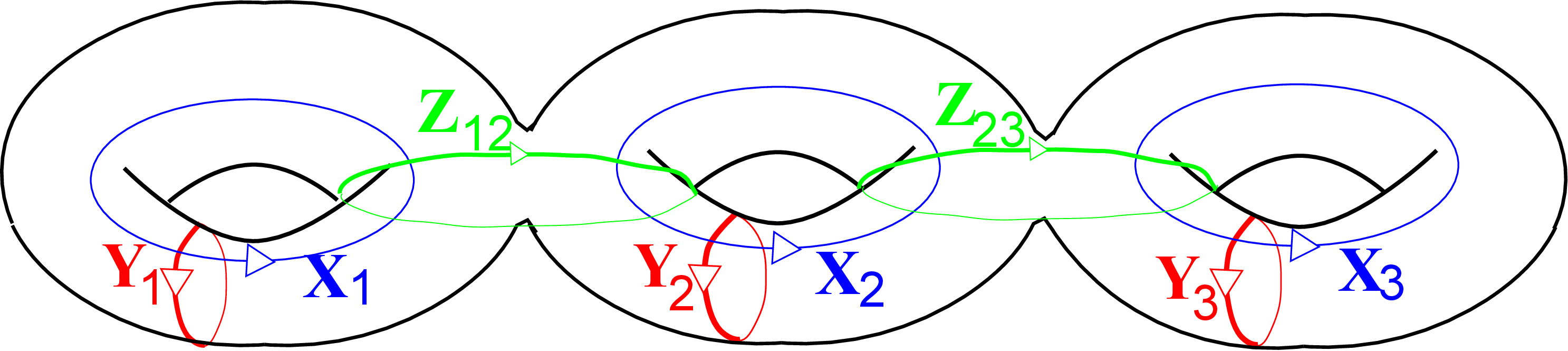}}
\caption{A set of Dehn twists generating the mapping class group and the symplectic group.} 
\label{gen_twists} 
\end{figure} 

\begin{figure}[!h] 
\centerline{\includegraphics[width=11cm]{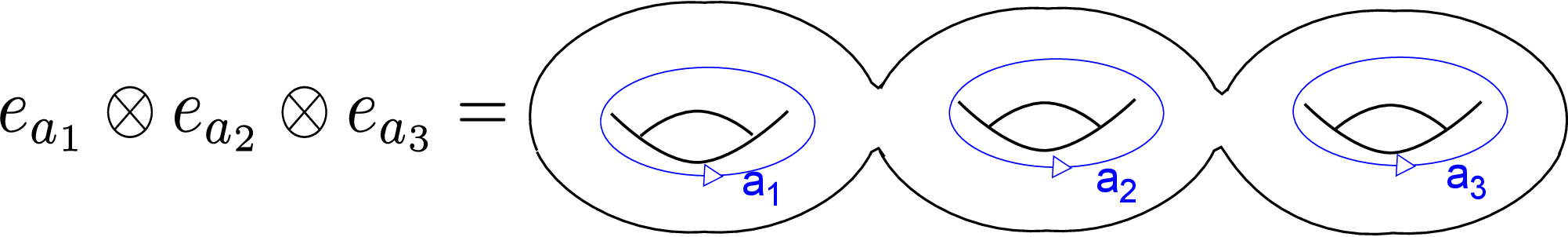}}
\caption{A basis for the abelian skein module of the genus $g$ handlebody. Here an integer $i$ in front of a ribbon means that we take $i$ parallel copies of it.} 
\label{base_tore_solid} 
\end{figure} 

\vspace{2mm} 
\begin{definition} Let $p\geq 2$ and $a,b$ be two integers. We define the Gauss sums by the formulas: \index{Gauss sums}
\begin{enumerate}
\item $G(a,b,p):=\sum_{k\in\zn}{A^{ak^2+bk}}\in \textbf{k}_p$ when $p$ is odd.
\item $G(a,b,2p):=\sum_{k\in\znn}{A^{ak^2+bk}}=2\sum_{k\in\zn}{A^{ak^2+bk}}\in \textbf{k}_p$ when $p$ is even.
\end{enumerate}
\end{definition}
\vspace{1mm}
\par The computation of the Gauss sums is detailed in \cite{BE}.

\begin{proposition}\label{matrix}
The expression of the matrices of the \weil representation on the generators $X_i,Y_i$ and $Z_{i,j}$ in the basis $\{e_{a_1}\otimes \ldots \otimes e_{a_g} | a_1,\ldots,a_g \in \zn \}$ of $U_p^{\otimes g}$ is given by the projective class of the following matrices: 
\begin{itemize}
\item $\pi_{p,1}(X)=(A^{2i^2}\delta_{i,j})_{i,j}$ and $\pi_{p,g}(X_i) = \mathds{1}^{\otimes (i-1)}\otimes \pi_p^1(X) \otimes \mathds{1}^{\otimes (g-i)}$.
\item  $\pi_{p,g}(Z_{i,j})(e_{a_1}\otimes\ldots\otimes e_{a_g}) = A^{(a_i-a_j)^2}(e_{a_1}\otimes\ldots\otimes e_{a_g})$.
\item $ \pi_{p,1}(Y) = \left\{ \begin{array}{ll}  \frac{G(1,0,N)}{N} (A^{-(i-j)^2})_{i,j}, & \mbox{ when } p \mbox{ is odd.}\\ \frac{G(1,0,2N)}{2N} (A^{-(i-j)^2})_{i,j}, & \mbox{ when } p \mbox{ is even.} \end{array} \right. $ \\   $ \pi_{p,g}(Y_i)=  \mathds{1}^{\otimes (i-1)}\otimes \pi_{p,1}(Y) \otimes \mathds{1}^{\otimes (g-i)}$.
\end{itemize}
These generating matrices are unitary (they verify $\bar{U}^TU=\mathbf{1}$ where $\bar{U}=\left( \bar{U_{i,j}}\right)_{i,j}$ is defined by the involution of $\textbf{k}_p$ sending $A$ to $A^{-1}$) so are the \weil representations. 
\end{proposition}
\vspace{2mm}
\begin{proof}
If $\phi \in \Mod(\Sigma_g)$ can be extended to a homeomorphism $\Phi$ of the handlebody $H_g$, the action of $\Phi$ on $\mathcal{T}_p(H_g)\cong U_p^{\otimes g}$ defines an operator which satisfies the Egorov identity $\eqref{egorov}$ so is projectively equal to $\underline{\pi}_{p,g}(\phi)$. The generators $X_i$ and $Z_{i,j}$ are such homeomorphisms and Figure $\ref{rep_weil}$ shows how we compute their action on the basis.
\vspace{1mm}\par
\begin{figure}[!h] 
\centerline{\includegraphics[width=11cm]{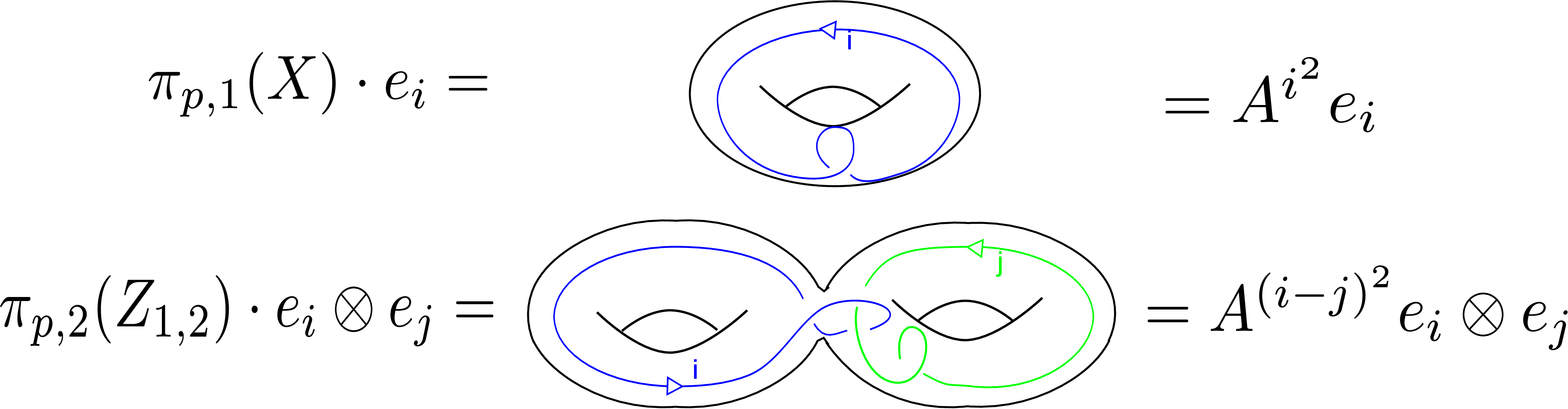} }
\caption{The computation of the matrices associated to ${\pi}_{p,1}(X)$ and ${\pi}_{p,2}(Z_{1,2})$.} 
\label{rep_weil} 
\end{figure} 

\par Then choose a Heegaard splitting of the sphere  $H_g\bigcup_{\phi} H_g \cong S^3$ with $\phi \in \Mod (\Sigma_g)$. This splitting determines a pairing  $\widetilde{\mathcal{T}}_p(H_g)\times \widetilde{\mathcal{T}}_p(H_g) \rightarrow \widetilde{\mathcal{T}}_p(S^3)\cong k'_p$. The associated  bilinear pairing $\left( \cdot,\cdot \right)_p^H : U_p^{\otimes g} \otimes U_p^{\otimes g} \rightarrow \textbf{k}_p$ is called the \textit{Hopf pairing}. Figure $\ref{hopf_pairing}$ shows that: 
$$ \left(e_{a_1}\otimes \ldots \otimes e_{a_g}, e_{b_1}\otimes \ldots \otimes e_{b_g} \right)_p^H = A^{-2\sum_i{a_ib_i}}$$
Thus the Hopf pairing is non degenerate.

\begin{figure}[!h] 
\centerline{\includegraphics[width=11cm]{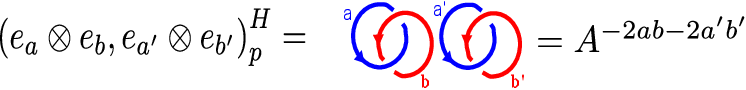}}
\caption{The computation of the matrix associated to the Hopf pairing when $g=2$.} 
\label{hopf_pairing} 
\end{figure} 

\par 
The dual of $\pi_{p,g}(X_i)$ for $\left< \cdot,\cdot \right> ^H$   satisfies the Egorov identity $\eqref{egorov}$, so is projectively equal to $\underline{\pi}_{p,g}(Y_i)$. If $\pi_{p,1}(Y)$ is the dual of $\pi_{p,1}(X)$ for $\left(\cdot ,\cdot \right) ^H$ , the previous expression of $\pi_{p,g}(X_i)$  implies that its dual for the Hopf pairing is  $\mathds{1}^{\otimes (i-1)}\otimes \pi_{p,1}(Y) \otimes \mathds{1}^{\otimes (g-i)}$.  
\par To compute the matrix of $\pi_{p,1}(Y)$, we remark that the matrix $S=\left( A^{-2ij} \right)_{i,j}$ of the Hopf pairing has inverse $S^{-1}=\frac{1}{p}\bar{S}=\frac{1}{p}(A^{2ij})_{i,j}$. A direct computation gives:
$$ \pi_{p,1}(Y) = S \pi_{p,1}(X) S^{-1} = 
\left\{ \begin{array}{ll} \frac{G(1,2(j-i),p)}{p} = \frac{G(1,0,p)}{p} (A^{-(i-j)^2})_{i,j}, &\rm{~ when ~}p\rm{ ~is~ odd;} 
\\ \frac{G(1,2(j-i),2p)}{2p} = \frac{G(1,0,2p)}{2p} (A^{-(i-j)^2})_{i,j}, &\rm{~ when~ }p\rm{ ~is~ even.}
\end{array} \right.$$

\end{proof}
\vspace{2mm}
\begin{remark}
\begin{enumerate}
\item When $p$ is even and $A=\exp \left(-\frac{i\pi}{p} \right)$,  the projective representations we defined here coincide with the ones from \cite{Fu2} and \cite{Go} coming from theta functions. 
\item When $p$ is odd or when $g=1$ and $p$ is even, the Weil representations lift to linear representations of $\Sp$ and $SL_2(\mathbb{Z}/2p\mathbb{Z})$ respectively (see \cite{AR} for a proof and \cite{Koju_thesis} for a proof that the matrices $\pi_{p,g}(X_i), \pi_{p,g}(Y_i)$ and $\pi_{p,g}(Z_{i,j})$, defined in Proposition $\ref{matrix}$ define an explicit lift).
\par When $p$ is even and $g\geq 2$, they lift to linear representations of $\widetilde{Sp_{2g}(\mathbb{Z})}$ a central extension of $Sp_{2g}(\znn)$ by $\mathbb{Z}/2\mathbb{Z}$ (see \cite{FP} for a proof and  \cite{Koju_thesis} for a proof that the matrices above define an explicit lift).
\par We will now consider these linear lifted representations and denote them by $\pi_{p,g}$.
\end{enumerate}
\end{remark}

\vspace{4mm}

\section{Decomposition of the \weil representations}

\vspace{1mm}
\par In this section we prove the three first points of the Theorem $\ref{main_theorem}$.   We first define:
\begin{eqnarray*}
U_p^{+,g} :=  \Span\{e_{a_1}\otimes \ldots \otimes e_{a_g}+e_{-a_1}\otimes \ldots \otimes e_{-a_g}| a_1,\ldots, a_g \in \zn\} \\
U_p^{-,g} :=  \Span\{e_{a_1}\otimes \ldots \otimes e_{a_g}-e_{-a_1}\otimes \ldots \otimes e_{-a_g}| a_1,\ldots, a_g \in \zn\}
\end{eqnarray*}

\vspace{2mm}
\begin{lemma}\label{point3}
\label{point3}
The submodules  $U_p^{+,g}$ and $U_p^{-,g}$ are $\pi_{p,g}$-stable. 
\end{lemma}
 
\begin{proof} A direct computation shows that the submodules $U_p^{g,+}$ and $U_p^{g,-}$ are stabilized by $\pi_{p,g}(X_i), \pi_{p,g}(Y_i)$ and $\pi_{p,g}(Z_{i,j})$. We can also remark that the involution acting on the reduced skein module by changing the orientation of a framed link, commutes with the image of $\pi$. The modules $U_p^{\pm,g}$ correspond to its two eigenspaces.

\end{proof}

\vspace{3mm}
\par   Let $a, b\geq 2$   be two coprime non negative integers with $b$ odd, and let $u$ and $v$ be odd integers such that $au+bv=1$ in the case where $a$  is odd  and such that $2au+bv=1$ if $a$ is even and $b$ is odd. We define a ring isomorphism $\mu : \textbf{k}_{ab}\rightarrow \textbf{k}_a\otimes \textbf{k}_b$ by $\mu(A)=(A^{vb},A^{au})$ if $a$ is odd and $\mu(A)=(A^{vb},A^{2au})$ if $a$ is even,   which turns $U_a^{\otimes g}\otimes U_b^{\otimes g}$ into a $\textbf{k}_{ab}$-module. We also denote by $f:\mathbb{Z}/a\mathbb{Z}\times \mathbb{Z}/b\mathbb{Z} \rightarrow \mathbb{Z}/{ab}\mathbb{Z}$ the  bijection sending $(x,y)$ to $xv+yu$ when $a$ is odd and to $xv+2yu$ when $a$ is even.
\vspace{2mm} The following lemma was shown in \cite{KR}, we give a more explicit proof.
\begin{lemma}[\cite{KR}]\label{point1}
 The isomorphism of $\textbf{k}_{ab}$-module $\psi :U_a^{\otimes g}\otimes U_b^{\otimes g} \rightarrow U_{ab}^{\otimes g}$ defined by $$\psi((e_{a_1}\otimes\ldots \otimes e_{a_g})\otimes (e_{b_1}\otimes\ldots \otimes e_{b_g}))=e_{f(a_1,b_1)}\otimes\ldots\otimes e_{f(a_g,b_g)}$$ makes the following diagram commute for all $\phi\in \ssp$ (resp for all $\phi\in \sph$ when $a$ is even): 
\begin{diagram}
U_a^{\otimes g}\otimes U_b^{\otimes g} &\rTo^{\psi} &U_{ab}^{\otimes g}\\
_{\pi_{a,g}(\phi)\otimes \pi_{b,g}(\phi)}\uTo & &\uTo_{\pi_{ab,g}(\phi)}\\
U_a^{\otimes g}\otimes U_b^{\otimes g} &\rTo^{\psi} &U_{ab}^{\otimes g}
\end{diagram}  
\end{lemma}
\vspace{4mm}
\begin{proof}

We note  $(A_1,A_2):=(A^{vb},A^{au})$ when $a$ and $b$ are odd and $(A_1,A_2)=(A^{vb},A^{2au})$ when $a$ is even. It is enough to show the commutativity of the diagram for $\phi=X_i,Y_i$ and $Z_{i,j}$. For $\phi=X_i$, we compute:
\begin{multline*}
\psi\left(\pi_{a,g}(X_i)\otimes \pi_{b,g}(X_i)((e_{a_1}\otimes\ldots\otimes e_{a_g})\otimes(e_{b_1}\otimes\ldots\otimes e_{b_g}))\right)=
\\  \psi \left( A_1^{a_i^2}A_2^{b_i^2} ((e_{a_1}\otimes\ldots\otimes e_{a_g})\otimes(e_{b_1}\otimes\ldots\otimes e_{b_g}))\right)=
\\  A^{f(a_i,b_i)^2} (e_{f(a_1,b_1)}\otimes \ldots \otimes e_{f(a_g,b_g)})
\end{multline*}
\vspace{2mm}
\par Then for $\phi=Y_i$, we note $c_p=\frac{G(1,0,p)}{p}$ when $p$ is odd and $c_p=\frac{G(1,0,2p)}{2p}$ when $p$ is even:
\begin{multline*}
\psi\left(\pi_{a,g}(Y_i)\otimes \pi_{b,g}(Y_i)((e_{a_1}\otimes\ldots\otimes e_{a_g})\otimes(e_{b_1}\otimes\ldots\otimes e_{b_g}))\right)
\\ = \psi \left( c_a c_b \sum_{\stackrel{k\in \mathbb{Z}/a\mathbb{Z}}{ l\in \mathbb{Z}/b\mathbb{Z}}}{A_1^{-(a_i-k)^2}A_2^{-(b_i-l)^2} ((e_{a_1}\otimes\ldots\otimes e_k \otimes \ldots \otimes e_{a_g})\otimes(e_{b_1}\otimes\ldots\otimes e_l \otimes \ldots\otimes e_{b_g}))}\right)
\\ = \psi(c_a c_b) \sum_{m\in \mathbb{Z}/ab\mathbb{Z}}{A^{-(f(a_i,b_i)-m)^2} (e_{f(a_1,b_1)}\otimes \ldots \otimes e_m\otimes \ldots \otimes e_{f(a_g,b_g)})}
\end{multline*}
where we made the change of variable $m=f(k,l)$ to pass to the last line. We conclude by noticing that $\psi(c_a c_b)=c_{ab}$ which is equivalent to $\psi(G(1,0,a)G(1,0,b))=G(1,0,ab)$ when $a$ is odd and  $\psi(G(1,0,2a)G(1,0,b))=G(1,0,2ab)$ when $a$ is even.
\vspace{2mm}
\par Finally for $\phi=Z_{i,j}$: 
\begin{multline*}
\psi\left(\pi_{a,g}(Z_{i,j})\otimes \pi_{b,g}(Z_{i,j})((e_{a_1}\otimes\ldots\otimes e_{a_g})\otimes(e_{b_1}\otimes\ldots\otimes e_{b_g}))\right)
\\ = \psi \left( A_1^{(a_i-a_j)^2}A_2^{(b_i-b_j)^2} ((e_{a_1}\otimes\ldots\otimes e_{a_g})\otimes(e_{b_1}\otimes\ldots\otimes e_{b_g}))\right)
\\ = A^{f(a_i,b_i)^2} (e_{f(a_1,b_1)}\otimes \ldots \otimes e_{f(a_g,b_g)})
\end{multline*}
\end{proof}
\vspace{2mm}
\begin{remark}
This lemma also follows from (\cite{MOO},  Proposition $2.3$) where it is showed that the $3$-manifold invariant coming from the abelian TQFT at level $ab$, with $a$ coprime to $b$, is the product of the ones in level $a$ and $b$. We can then conclude using the same argument that in \cite{BHMV2}.

\end{remark}

\vspace{5mm}
\par  Let $r$ be a prime number and $n\geq 0$ if $r$ is odd or $n\geq 1$ if $r=2$. Let $\bar{U}^{\otimes g}_{r^{n}}$  be the submodule of $U^{\otimes g}_{r^{n+2}}$ spanned by the vectors $g_{a_1}\otimes\ldots\otimes g_{a_g}$ where  $g_i:=\sum_{0\leq k \leq r-1} e_{r(i+kr^{n})}$. 
\vspace{2mm}
\begin{lemma}\label{point2}
The submodule $\bar{U}^{\otimes g}_{r^n} $ is stabilized by $\pi_{{r^{n+2}},g}$. Moreover the isomorphism of $\textbf{k}_{r^{n+2}}$-modules $\psi : U_{r^{n}}^{\otimes g}\rightarrow \bar{U}^{\otimes g}_{r^{n}}$ sending $e_{a_1}\otimes\ldots\otimes e_{a_g}$ to $g_{a_1}\otimes\ldots\otimes g_{a_g}$ makes the following diagram commute for all $\phi\in \ssp$ ( for all $\phi\in \sph$ when $r=2$ respectively):
\begin{diagram}
\GL(U^{\otimes g}_{r^{n+2}})^{\bar{U}^{\otimes g}_{r^{n}}}&\rTo^{\pi_{r^{n+2},g}(\phi)} &\GL(U^{\otimes g}_{r^{n+2}})^{\bar{U}^{\otimes g}_{r^{n}}}\\
\uInto & &\uInto\\
\GL(U^{\otimes g}_{r^{n}}) &\rTo^{\pi_{r^{n},g}(\phi)} &\GL(U^{\otimes g}_{r^{n}})
\end{diagram}
\end{lemma}
\vspace{4mm}
\begin{proof}
We generalize an argument of  \cite{CNS} to even levels to show that $\bar{U}^{\otimes g}_{r^n}$ is $\pi_{{r^{n+2}},g}$-stable. Denote by $I$ the principal ideal $I:=r^{n+1}H_1(\Sigma^g,\mathbb{Z}/r^{n+2}\mathbb{Z})$ of $H_1(\Sigma^g,\mathbb{Z}/r^{n+2}\mathbb{Z})$ and by $D$  the subgroup $D:=(I\times I, 0)$ of $\mathcal{H}_{r^{n+2},g}$. Since $I^2=\{ 0\} $ and $I$ is an ideal, $D$ is a subgroup of $\mathcal{H}_{r^{n+2},g}$ stable under the action of $\ssp$. We deduce from the Egorov identity that the space $\{ v\in U_{r^{n+2}}^{\otimes g} | \Add_p(\phi)v=v, \forall \phi \in D\}$ is preserved by $\pi_{{r^{n+2}},g}$. We now easily show that this space is  $\bar{U}^{\otimes g}_{r^n}$.
\vspace{3mm}
\par We then verify the commutativity of the diagram for $\phi=X_i, Y_i$ and $Z_{i,j}$. When $\phi=X_i$ we have:
$$
\pi_{{r^{n+2}},g}(X_i) (g_{a_1}\otimes \ldots \otimes g_{a_g}) = A^{(ri)^2} (g_{a_1}\otimes \ldots \otimes g_{a_g})
 = \mu(A)^{i^2} (g_{a_1}\otimes \ldots \otimes g_{a_g})
$$
\vspace{1mm}
\par When $\phi=Y_i$ we have:
\begin{multline*}
\pi_{{r^{n+2}},g}(Y_i) (g_{a_1}\otimes \ldots \otimes g_{a_g}) = c_{r^{n+2}}\sum_{x\in \mathbb{Z}/r^{n+2}}\sum_{k\in \mathbb{Z}/r\mathbb{Z}}{A^{-(r(a_i+kr^n)-x)^2}g_{a_1}\otimes \ldots \otimes e_x\otimes \ldots \otimes e_{a_g}}
\\ = c_{r^{n+2}}\sum_{x\in\mathbb{Z}/r^{n+2}\mathbb{Z}} A^{-x^2-xra_i-r^2a_i^2}\left( \sum_{k\in\mathbb{Z}/r\mathbb{Z}}{(A^{2r^(n+1)x})^k}\right) g_{a_1}\otimes \ldots \otimes e_x \otimes \ldots \otimes g_{a_g}
\\  = r c_{r^{n+2}}\sum_{y\in \mathbb{Z}/r^{n+1}\mathbb{Z}} (A^{r^2})^{-(y-a_i)^2} g_{a_1}\otimes \ldots \otimes e_{ry}\otimes \ldots \otimes g_{a_g}
\\ = r c_{r^{n+2}} (\mu(A))^{-(z-a_i)^2} \sum_{z\in \mathbb{Z}/r^n \mathbb{Z}} g_{a_1}\otimes \ldots \otimes g_z \otimes \ldots \otimes g_{a_k}
\end{multline*}
We verify that $\mu(c_{r^n})=r c_{r^{n+2}}$ to conclude in this case. Finally when $\phi=Z_{i,j}$:
\begin{multline*}
\pi_{{r^{n+2}},g}(Z_{i,j}) (g_{a_1}\otimes \ldots \otimes g_{a_g})
\\ = \sum_{k,l\in \mathbb{Z}/p\mathbb{Z}} A^{(r(a_i+kr^n)-p(a_j+lr^n))^2} (g_{a_1}\otimes \ldots e_{r(a_i+kr^n)} \otimes \ldots \otimes e_{r(a_j+lr^n)}\otimes \ldots \otimes g_{a_g})
\\ =  \sum_{k,l\in \mathbb{Z}/p\mathbb{Z}} (A^{r^2})^{(a_i-a_j)^2} (g_{a_1}\otimes \ldots e_{r(a_i+kr^n)} \otimes \ldots \otimes e_{r(a_j+lr^n)}\otimes \ldots \otimes g_{a_g})
\\ = (\mu(A)^{(a_i-a_j)^2} (g_{a_1}\otimes \ldots \otimes g_{a_g})
\end{multline*}

\end{proof}
Let  $W_{r^{n+2}}$ be the submodule of $U_{r^n}$ orthogonal for the invariant form turning $\{e_0,\ldots, e_{r^{n+2}-1}\}$ into an orthogonal basis. It is freely generated by the vectors $e_i$ when $r$ does not divide $i$ and by the vectors $e_{ri-r(i+k+r^n)}$ for $i\in \{0,\ldots,r^n-1 \}$ and $k\in\{1,\ldots, r-1 \}$.
\par The orthogonal of $\bar{U}_{r^n}^{\otimes g}$ in $U_{r^{n+2}}^{\otimes g}$ is isomorphic to $W_{r^{n+2}}^{\otimes g}$ and is stabilized by $\pi_{{r^{n+2}},g}$. So are the two submodules $W_{r^{n+2}}^{g,\pm} := W_{r^{n+2}}^{\otimes g} \bigcap U_{r^{n+2}}^{g,\pm}$.
\vspace{4mm}

\section{Irreducibility of the factors}
\subsection{The genus one cases}

\vspace{2mm}
\par The goal of this section is to extend Kloosterman's work \cite{Kl} to even levels.
\vspace{2mm}\par
 When $g=1$ the strategy for the proof lies on the computation of the following Kloosterman's sums:
\begin{eqnarray}\label{sums}
\textbf{S}_p &:= \frac{1}{|\Sp|}\sum_{\phi\in \Sp}{|\tr(\pi_p(\phi))|^2}, & \mbox{ when } p \mbox{ is odd.} \\
\textbf{S}_{2p} &:= \frac{1}{|\Spp|}\sum_{\phi\in \Spp}{|\tr(\pi_p(\phi))|^2} ,& \mbox{ when } p \mbox{ is even.}
\end{eqnarray}
 It is a classical fact that if this sum is equal to the number of component in a decomposition of $\pi_p$ then each factors appearing in this decomposition is irreducible and they are pairwise distinct (see  \cite{Serre}, chapter $2$). 
\begin{lemma}\label{tata}
If $a$ is prime to $b$ then $\textbf{S}_{ab}=\textbf{S}_a \times \textbf{S}_b$ if they are both odd and $\textbf{S}_{2ab}=\textbf{S}_{2a}\times \textbf{S}_b$ if $a$ is even.
\end{lemma}
\begin{proof}
This follows from the fact that we have a group isomorphism $SL_2(\mathbb{Z}/ab\mathbb{Z}) \cong SL_2(\mathbb{Z}/a\mathbb{Z}) \times SL_2(\mathbb{Z}/b\mathbb{Z})$  together with Proposition $\ref{point1}$.
\end{proof}
\vspace{3mm}
\par In \cite{Kl} Kloosterman showed that for an odd prime $r$ and $n\geq 1$ then 
$ \textbf{S}_{r^n}=n+1$.
 Thus, to complete the proof of Theorem $\ref{main_theorem}$ it remains to show the following:

\begin{proposition}\label{prop_sums}
For $n\geq 1$, we have:
$$\textbf{S}_{2^n}=n-1$$
\end{proposition}

 Since the summand $|\tr (\pi_{2^n}(\phi))|^2$ only depends on the conjugacy class of $\phi$ we will first make a complete study of the conjugacy classes of $\Sppp$.  Then we will compute the characters of the \weil representations on representatives of each conjugacy classes.

\vspace{4mm}
\subsubsection{Conjugacy classes of $\Sppp$}

\vspace{2mm}
\par 
We begin by defining three invariants of the conjugacy classes which almost classify the conjugacy classes: 

\vspace{2mm}
\par \begin{definition}
For $A\in \Sppp$ there exists a unique integer $l\in\{0,\ldots, n\}$ and $x\in \{ 0,\ldots,2^l-1\}$ such that:
\begin{equation*}
A \equiv  x \mathds{1} +2^l U_1 \; \pmod{2^n}
\end{equation*} 
for some matrix $U_1$ which reduction modulo $2$ is neither the identity, nor the null matrix. We define a third integer 
$$\tau:=\left\{
\begin{array}{ll}
\tr(A)\in\mathbb{Z}/2^n\mathbb{Z}, & \mbox{ when }l=0.
\\ \det (U_1)\in\mathbb{Z}/2^{n-l}\mathbb{Z}, & \mbox{ when }l\geq 1.
\end{array} \right.
$$ 
Note that $\det(U)=1 (\mbox{mod }2^n)$ implies that $x^2=1 (\mbox{mod }2^l)$ hence if $l=1$ then $x=1$, when $l=2$ then $x= 1$ or $3$, when $l\geq 3$ we have four choices: $x=1,2^l-1, 2^{l-1}+1$ or $2^{l-1}-1$.
\par Let us denote by $C(x,l,\tau)$ the set of matrices of $\Sppp$ having $x,l$ and $\tau$ as invariants. Clearly $C(-1,l,\tau)=-C(1,l,\tau)$ and $C(2^{l-1}-1,l,\tau) =-C(2^{l-1}+1,l,\tau)$, thus we only need to study the conjugacy classes of $C(x,l,\tau)$ when $x=1$ or $x=2^{l-1}+1$.
\end{definition}
\par As example, the matrices with $l=0$ are the matrices which are not equal to the identity matrix modulo $2$ whereas those with $l=n$ are the four scalar matrices.

\vspace{3mm}
\par \begin{definition} We define the following representatives of $C(x,l,\tau)$, where $c_1$ will denote an odd number: 
\begin{itemize}
\item $l=0$, $A_0(\tau,c_1):= \begin{pmatrix} 1& c_1^{-1}(\tau - 2) \\ c_1 & \tau -1 \end{pmatrix}$.
\item $l\geq 1$, $x=1$, $A_l(\tau, c_1) := \begin{pmatrix} 1 & c_1^{-1}2^l \tau \\ c_1 2^l & 1+2^l \tau \end{pmatrix}$.
\item $l\geq 3$, $x=1+2^{l-1}$, $B_l(\tau,c_1) := \begin{pmatrix} 1+2^{l-1} & -c_1^{-1}2^l\tau \\ 2^l c_1 & 1+2^{l-1}-(1+2^{l-1})^{-1}(2^l+2^{2l-2}+2^{2l}\tau) \end{pmatrix}$.
\end{itemize}
Similar representative for $x=-1$ and $x=2^{l-1}-1$ are given by taking  $-A_l$ and $-B_l$.
\end{definition}
\vspace{2mm}
\begin{proposition}\label{prop_conj}
\par Each set $C(x,l,\tau)$ contains $1,2$ or $4$ conjugacy classes each containing a matrix $\pm A_l(\tau,c_1)$ or $\pm B_l(\tau,c_1)$ for a suitable choice of $c_1$. The following table gives for every $l,x,\tau$ a set of $1,2$ or $4$ representatives and the cardinal $m(A)$ of the corresponding conjugacy classes: 
\vspace{2mm} \par \begin{center}
\begin{tabular}{|m{2cm}||c|c|l|}
\hline
$l$ and $x$ & $\tau$ & Representatives  of $C(x,l,\tau)$ & $m(A)$ \\ \hline
$l=0$ & $\tr(U)=\tau$ is odd &  $A_0(\tau,1)$ & $2^{2n-1}$ \\ 
\cline{2-4}
 & $\tr(U)=\tau=2 \pmod{4}$ & $A_0(\tau,1),A_0(\tau,3),A_0(\tau,5),A_0(\tau,7)$ & $3\cdot 2^{2n-4}$ \\
\cline{2-4}
 & $\tr(U)=\tau=0  \pmod{4}$ & $A_0(\tau,1),A_0(\tau,3)$ & $3\cdot 2^{2n-3}$ \\
\hline
$l=1$ \newline and $x=1$ & $\tau=1 \pmod{8}$ & $A_1(\tau,1),A_1(\tau,3),A_1(\tau,5),A_1(\tau,7)$ & $3\cdot 2^{2n-6}$ \\
\cline{2-4}
 & $\tau=3,5,7  \pmod{8} $ & $A_1(\tau,1),A_1(\tau,\tau)$ & $3\cdot 2^{2n-5}$ \\
\cline{2-4}
& $\tau=2,4,6  \pmod{8} $ & $A_1(\tau,1),A_1(\tau,3)$ & $3\cdot 2^{2n-5}$ \\
\cline{2-4}
 & $\tau=0  \pmod{8} $ & $A_1(\tau,1),A_1(\tau,3),A_1(\tau,5),A_1(\tau,7)$ & $3\cdot 2^{2n-6}$ \\
\hline
$2\leq l \leq n-3$ and $x=1$ & $\tau=1,4,5  \pmod{8}$ & $A_l(\tau,1),A_l(\tau,3)$ & $3\cdot 2^{2n-2l-3}$ \\
\cline{2-4}
  & $\tau=3,7  \pmod{8}$ & $A_l(\tau,1)$ & $3\cdot 2^{2n-2l-2}$ \\
\cline{2-4}
 & $\tau=2  \pmod{8}$ & $A_l(\tau,1),A_l(\tau,5)$ & $3\cdot 2^{2n-2l-3}$ \\
\cline{2-4}
 & $\tau=0  \pmod{8}$ & $A_l(\tau,1),A_l(\tau,3),A_l(\tau,5),A_l(\tau,7)$ & $3\cdot 2^{2n-2l-4}$ \\
\hline
$l=n-2$\newline and $x=1$ & $\tau=0,1 \pmod{4}$ & $A_{n-2}(\tau,1), A_{n-2}(\tau,3)$ & $6$\\
\cline{2-4}
 & $\tau=2,3  \pmod{4}$ & $A_{n-2}(\tau,1)$ & $12$ \\
\hline
$l=n-1$ \newline and $x=1$ & $\tau=0  \pmod{2}$ & $A_{n-1}(0,1)$ & $3$ \\
\cline{2-4}
 & $\tau=1  \pmod{2}$ & $A_{n-1}(1,1) $ & $3$ \\
\hline
$3\leq l \leq n-1$\newline and \newline $x=1+2^{l-1}$ & $\tau$ odd & $B_l(\tau,1)$ & $2^{2n-2l-1}$ \\
\cline{2-4}
 & $\tau$ even & $B_l(\tau,1)$ & $3\cdot 2^{2n-2l-1}$ \\
\hline
$l=n$ & & $\mathds{1}, -\mathds{1}, (2^{n-1}+1)\mathds{1}$ and $(2^{n-1}-1)\mathds{1}$ & $1$\\
\hline
\end{tabular}
\end{center}
\end{proposition}
\vspace{4mm}
\par  Proposition $\ref{prop_conj}$ gives the complete description of the conjugacy classes of $\Sppp$.  The exact information needed for computing $\textbf{S}_{2^n}$ is summarized in the following: 
\vspace{1mm}
\begin{corollary}\label{corro_class}
For $A\in \Sppp$ we define $s(A)\in \{2l,\ldots, l+n\}$ to be the maximal $s$ for which $2^{s-l}$ divides $\tau$. Let $N(l,x)$, resp $N(l,x,s)$, be the number of matrices having $l,x$ (resp $s$) as invariants. We deduce from Theorem \ref{prop_conj} the following:
\begin{enumerate}
\item $N(0,1,0)=2^{3n-2}$.
\item For $1\leq s \leq n-1$, $N(0,1,s) =3\cdot 2^{3n-s-3}$.
\item $N(0,1,n)=3\cdot 2^{2n-2}$.
\item For $l\geq 1$,  $N(l,1,s)=3.2^{3n-l-s-3}$ if $s\neq l+n$ and $N(l,1, n+l)=3\cdot 2^{2n-2l-2}$. 
\item For $l\geq 2$, $N(l,-1)=3\cdot 2^{3n-3l-2}$.
\item For $l\geq 3$, $N(l,1+2^{l-1})=N(l,2^{l-1}-1)=2^{3n-3l}$.
\item $N(n,x)=1$.
\end{enumerate}
\end{corollary}

\vspace{4mm}
\par The proof of Proposition $\ref{prop_conj}$ will be deduced from the following:
\vspace{2mm}
\begin{lemma}\label{lemma_imp}
Let $U=\begin{pmatrix} a&b\\c&d \end{pmatrix}$ and $U'=\begin{pmatrix} a'&b'\\c'&d' \end{pmatrix}$ be two matrices of $C(x,l,\tau)$. If  $l=0$, we suppose that $c$ and $c'$ are odd. If $l\geq 1$, writing $U=x \mathds{1} +\begin{pmatrix} a_1&b_1\\c_1&d_1 \end{pmatrix}$ we suppose that $c_1$ and $c_1'$ are odd. Note that each conjugacy class contains an element satisfying these conditions. We define $E_{U,U'}$ the following equation: 
\begin{eqnarray*}
c_1 x^2+(a_1-d_1)xy-b_1y^2 \equiv  c_1' &  \pmod{2^{n-l}}, & \rm{~ when~ } l\geq 1; 
\\ c x^2+(a-d)xy-by^2 \equiv  c' &  \pmod{2^{n}}, & \rm{~ when~ } l=0.
\end{eqnarray*}
Then we have the two following properties: 
\begin{enumerate}
\item  The matrix $U$ is conjugate to $U'$ if and only if $E_{U,U'}$ has solutions.
\item  If $k$ is the number of solutions of $E_{U,U}$ then the conjugacy class of $U$ has $m(U)=\frac{1}{k}3\cdot 2^{3n-3l-2}$ elements.
\end{enumerate}
\end{lemma}
\vspace{3mm}
\par Once this Lemma proved, the proof of  Theorem $\ref{prop_conj}$ will follows from the study of the equations $E_{U,U'}$. We will need the  Hensel's Lemma (see \cite{Che}, section $3.2$) which states that if $n\geq 1$, $x_0\in\znnn$ and $P\in\mathbb{Z}[x]$ is a polynomial such that $P(x_0) \equiv 0 \;  \pmod{2^n}$ and $P'(x_0)$ is odd, then there exists a unique element $\tilde{x_0}\in\mathbb{Z}/2^{n+1}\mathbb{Z}$ such that $\tilde{x_0} \equiv x_0  \pmod{2^n})$ and $P(\tilde{x_0}) \equiv 0  \pmod{2^{n+1}}$.

\vspace{2mm}
\begin{lemma}\label{lemme5}
\par Let $A\in \Sppp$, then there exists exactly $8$ matrices $\tilde{A}\in SL_2(\mathbb{Z}/2^{n+1}\mathbb{Z})$ such that $\tilde{A} \equiv A  \pmod{2^n}$.
\end{lemma}
\vspace{2mm}
\begin{proof}
\par Let $A=\begin{pmatrix} a&b\\c&d\end{pmatrix}$. 
 Then  at least one entry of $A$ must be odd. Suppose $c$ is odd. There are exactly $8$ ways to lift $a,c$ and $d$ into elements $\tilde{a},\tilde{c},\tilde{d}$ in $\mathbb{Z}/2^{n+1}\mathbb{Z}$. Using Hensel's Lemma to the polynomial $P(b):=-\tilde{c}b+\tilde{a}\tilde{d}-1$ we show that for each of this $8$ choices, there is exactly one way to lift $b$  in $\mathbb{Z}/2^{n+1}\mathbb{Z}$ such that the corresponding matrix $\tilde{A}$ lies in $SL_2(\mathbb{Z}/2^{n+1}\mathbb{Z})$.
\end{proof}
\vspace{3mm}
\par Note that this lemma easily implies by induction that the cardinal of $\Sppp$ is $3\cdot 2^{3n-2}$.
\vspace{4mm}
\par \begin{proof}[Proof of Lemma \ref{lemma_imp}]
Suppose that $X=\begin{pmatrix} x_1&y_1\\x_2 & y_2\end{pmatrix}\in \Sppp$ is such that $XUX^{-1}=U'$. 
A simple computation shows that $XUX^{-1}$ has the form $XUX^{-1}=\begin{pmatrix} * & * \\ cy_2^2+(a-d)x_2y_2-bx_2^2 & *\end{pmatrix}$. Thus $(y_2,x_2)$ is solution of $E_{U,U'}$.
\vspace{2mm}
\par Conversely, let $(y_2,x_2)$ be solution of $E_{U,U'}$. The equality $XU=U'X$ is equivalent to the following equations:
 \begin{eqnarray}
x_1a +cy_1 &=& a'x_1+b'x_2 \label{un}
\\ x_1b+y_1d &=& a'y_1+b'y_2 \label{deux}
\\ x_2 a +cy_2 &=& c'x_1 +d'x_2 \label{trois}
\\ x_2 b +d y_2 &=& c' y_1 +d' y_2 \label{quatre}
\end{eqnarray} 
The equations $\eqref{trois}$ and $\eqref{quatre}$ completely determine the values of $x_1$ and $y_1$, so of $X$,  modulo $2^{n-l}$. Direct computations show that this $X$ is in $SL_2(\mathbb{Z}/2^{n-l}\mathbb{Z})$ and verifies $\eqref{un}$ and $\eqref{deux}$.


 Thus an element  $X$ in the stabilisator  $\Stab(U)$ of $U$ is completely determined modulo $2^{n-l}$ by a solution of $E_{U,U}$. Using Lemma $\ref{lemme5}$, we see that there are exactly $2^{3l}$ ways to lift such a matrix in  $\Sppp$. So, if $k$ is the number of solutions of $E_{U,U}$ then $|\Stab(U)|=k2^{3l}$. The class formula concludes the proof.

\end{proof}
\vspace{4mm}
\par It remains to compute the number of solutions of the equations $E_{U,U'}$. 
\begin{lemma}\label{lemme4}
Let $n\geq 1$ and $A,B,C,D$ four integers so that $ABD$ is odd. Let $E_n$ be the following equation: 
\begin{equation*}
Ax^2+Bxy+Cy^2 \equiv D \;   \pmod{2^n}
\end{equation*}
Then $E_n$ has $2^{n-1}$ solutions if $C$ is even and $3\cdot 2^{n-1}$ solutions if $C$ is odd.
\end{lemma}

\vspace{2mm}
\begin{proof}
We show the result by induction on $n$ using Hensel's Lemma.

\end{proof}
\vspace{2mm}
\begin{lemma}\label{lemme2}
\par Let $n\geq 1$ and $A,B,C,D$ be integers such that $A$ and $D$ are odd. Let $(E)$ be the following equation with variables $(x,y)$ both in $\Sp$ : 
\begin{equation*}
Ax^2+2Bxy+Cy^2 \equiv D \;  \pmod{2^n}
\end{equation*}
\vspace{1mm} \par We note $\Delta :=AC-B^2$. Then: 
\par \textbf{(1)} If $n=1$, $(E)$ has $2$ solutions.
\par \textbf{(2)} If $n=2$,  when $\Delta \equiv 2,3  \pmod{4}$ then $(E)$ has $4$ solutions. When $\Delta \equiv 0,1 \pmod{4}$ then $(E)$ has $8$ solutions if $AD\equiv 1 \pmod{4}$ and $0$ otherwise.
\par \textbf{(3)} If $n\geq 3$, we have the following cases: 
\begin{itemize}
\item \textbf{(a)} If $\Delta \equiv  0 \pmod{8}$ then $(E)$ has $2^{n+2}$ solutions if $AD\equiv  1 \pmod{8}$ and $0$ otherwise.
\item \textbf{(b)} If $\Delta \equiv  2,4,6 \pmod{8}$ then $(E)$ has $2^{n+1}$ solutions if $AD\equiv  1 \pmod{8}$ or $AD\equiv  1+\Delta  \pmod{8}$ and $0$ otherwise.
\item \textbf{(c)} If $\Delta \equiv 1,5 \pmod{8}$ then $(E)$ has $2^{n+1}$ solutions if $AD\equiv  1 \pmod{8}$ or $AD\equiv  5 \pmod{8}$ and $0$ otherwise.
\item \textbf{(a)} If $\Delta \equiv  3,7  \pmod{8}$ then $(E)$ has $2^n$ solutions.
\end{itemize}
\end{lemma}

\vspace{2mm}
\par \begin{proof}
\par First we put $z=Ax+By$. The map from $\mathbb{Z}/2^n\mathbb{Z}\times\mathbb{Z}/2^n\mathbb{Z}$ to itself sending $(x,y)$ to $(z,y)$ is bijective as $A$ is odd and we remark that $(x,y)$ is solution of $(E)$ if and only if $(z,y)$ is solution of the following equation, say $(E')$: 
\begin{equation*}
z^2 +\Delta y^2 \equiv  AD \;  \pmod{2^n}
\end{equation*}
Thus $(E)$ and $(E')$ have the same number of solutions. The number of solutions of $(E')$ is easily computed using the  fact (see \cite{Dem}, proposition $5.13$) that if $a$ is an odd number and $n\geq 3$, then the equation $x^2\equiv a  \pmod{2^n}$ has $4$ solutions modulo $2^n$ if $a \equiv 1 \pmod{8}$ and $0$ otherwise.

\end{proof}

\vspace{2mm}
\begin{proof}[End of the Proof of Theorem \ref{prop_conj}]

We fix three invariants $l,x$ and $\tau$ and study the conjugacy classes of $C(l,x,\tau)$. Let us take two matrices $U,U'\in C(l,x,\tau)$. We can always conjugate them so that they verify the hypothesis of Lemma $\ref{lemma_imp}$. These two matrices are conjugate if and only if the set of solutions of $E_{U,U'}$ is not empty and the number of elements in the conjugacy class of $U$ is computed by using Lemmas $\ref{lemma_imp}$, $\ref{lemme2}$ and $\ref{lemme4}$.

\end{proof}

\vspace{4mm}
\subsubsection{Computation of the characters}
\vspace{3mm}
\par \begin{proposition}\label{even_traces}
Let $A\in \Sppp$ and $x,l,s$ be its associated invariants. The definition of $s$ has been given in  Corollary $\ref{corro_class}$ and will make sense now. The  trace  $\tr(\pi_{2^{n-1}}(A))$ is given by: 
\begin{enumerate}
\item If $l=0$, $|\tr(\pi_{2^{n-1}}(A))|^2=2^s$ if $0\leq s \leq  n-2$,  $\tr(\pi_{2^{n-1}}(A))=0$ if $s=n-1$ and $|\tr(\pi_{2^{n-1}}(A))|^2=2^{n-1}$ if $s=n$. 
\item If $1\leq l \leq n-2$ and $x=1$ then $|\tr(\pi_{2^{n-1}}(A))|^2=2^s$ when $2l\leq s \leq n+l-2$, $\tr(\pi_{2^{n-1}}(A))=0$ when $s=n+l-1$ and $|\tr(\pi_{2^{n-1}}(A))|^2=2^{n+l-1}$ if $s=n+l$.
\item If $l=n-1$ and $x=1$ then $\tr(\pi_{2^{n-1}}(A))=0$.
\item If $l=n$ and $x=1$ ($A=I_2$) then $|\tr(\pi_{2^{n-1}}(A))|^2=2^{2n-2}$.
\item If $2\leq l \leq n$ and $x=-1$ then $|\tr(\pi_{2^{n-1}}(A))|^2=4$.
\item If $3\leq l \leq n$ and $x=2^{l-1}+1$ then $|\tr(\pi_{2^{n-1}}(A))|^2=2^{2l-2}$.
\item If $3\leq l \leq n$ and $x=2^{l-1}-1$ then $|\tr(\pi_{2^{n-1}}(A))|^2=4$.
\end{enumerate}
\end{proposition}
\vspace{3mm}
\begin{lemma}\label{lemma_diag}
Let $a$ be an odd integer and $D_a:=\begin{pmatrix} a & 0 \\ 0& a^{-1} \end{pmatrix}\in \Sppp$. Then we have $\pi_{2^{n-1}}(D_a)=\epsilon (\delta_{ai,j})_{i,j}$ where $\epsilon$ is a scalar such that $|\epsilon|^2=1$.
\end{lemma}
\vspace{2mm}
\begin{proof} It is proved by a direct computation using the fact that $D_a=T^{-a}ST^{-a^{-1}}ST^{-a}S$ .

\end{proof}

\vspace{3mm}
\begin{proof}[Proof of Proposition $\ref{even_traces}$]
\par  First when $l=0$ or when $x=1$,  we can suppose that $A=\begin{pmatrix} 1&b\\c& 1+bc \end{pmatrix} =ST^{c}S^{-1}T^{-b}$ with $b=2^{s-l} b_1$, $c=2^{l}c_1$ where $b_1$ and $c_1$ are odd.
\vspace{2mm}
 \par A simple computation gives:
$$
\pi_{2^{n-1}}(A) = \beta^{\pm 3 +x}\frac{G(-1,0,2^n)^2}{2^{2n}}\left(\sum_k{A^{ck^2+2(j-i)k-bj^2}}\right)_{i,j}
$$
So:
\begin{equation*}
|\tr(\pi_{2^{n-1}}(A))|=\left| \left(\frac{G(-1,0,2^n)}{2^n}\right)^2\frac{G(c,0,2^{n})}{2}\frac{G(-b,0,2^{n})}{2} \right|
\end{equation*}
\vspace{1mm}
\par We conclude by using the fact that, if $x$ is odd and $s\in \{ 0, \ldots , n\}$ then (see \cite{BE}):
$$ |G(x2^s,0,2^n)|^2 = \left\{ 
\begin{array}{ll}
2^{s+n}, & \mbox{ when } s\leq n-2;
\\ 0 , & \mbox{ when } s=n-1;
\\ 2^n, & \mbox{ when } s=n.
\end{array}\right.
$$ 

\vspace{3mm}
\par  Then when $x=-1$ we can suppose $A=-\begin{pmatrix} 1&b\\c& 1+bc \end{pmatrix} =S^{-1}T^{c}S^{-1}T^{-b}$ with $b=2^{s-l} b_1$, $c=2^l c_1$ where $b_1$ and $c_1$ are odd. A similar computation gives: 
$$\pi_{2^{n-1}}(A)_{i,i} = \epsilon \left( \frac{G(-1,0,2^n)}{2^n} \right) ^2A^{-bi^2}\frac{G(c,4i,2^n)}{2}$$
where $\epsilon=\beta^{c-b-6}$ is a norm one scalar. 
The Gauss sum $G(c,4i,2^n)$ is not null if and only if $i\in \{ 0, 2^{n-2}\}$ when $l=n$, $2^{n-3}$ divides $i$ and $2^{n-2}$ does not when $l=n-1$ and $2^{l-1}$ divdes  $i$ when $2\leq l\leq n-3$.
\vspace{2mm}
\par We conclude by summing $\pi_{2^{n-1}}(A)_{i,i}$ over these $i$.


\vspace{3mm}
\par Now to compute the traces when $x=2^{l-1}\pm 1$, we write $A=\begin{pmatrix} a&b\\c& d \end{pmatrix}$ with $a$ odd and $c=2^l c_1$ with $c_1$ odd. We  use  the decomposition $A=ST^{ca^{-1}}SD_{-a}T^{-a^{-1}b}$ and  Lemma $\ref{lemma_diag}$ to find that: 
$$(\pi_{2^{n-1}}(A))_{i,i}=\epsilon' \left(\frac{G(-1,0,2^n)}{2^n}\right)^2 \frac{G(ca^{-1}, 2(a^{-1}-1)i, 2^n)}{2} A^{a^{-1}bi^2}$$

\par where $\epsilon'$ is a norm one scalar. We conclude by summing $\pi_{2^{n-2}}(A)_{i,i}$ over every $i$ and taking the norm.
\end{proof}

\vspace{5mm}
\subsubsection{The computation of the sum $\textbf{S}_{2^n}$}
\vspace{2mm}
\begin{proof}[Proof of Proposition \ref{prop_sums}]
\par  Set $\textbf{S}(x,l):=\sum_{A\in C(x,l)}{|T(A)|^2}$ and $\textbf{S}(l):=\sum_{A\in C(l)}{|T(A)|^2}$. By using Propositions $\ref{corro_class}$ and $\ref{even_traces}$ together, we compute the following sums:
\begin{enumerate}
\item $\textbf{S}(0)=2^{3n-2}+3\cdot 2^{3n-3}(n-1)$.
\item $\textbf{S}(1, l)=3\cdot 2^{3n-l-3}(n-l)$ if $1\leq l \leq n-2$.
\item $\textbf{S}(-1,l)=3\cdot 2^{3n-3l}$ if $2\leq l \leq n-1$.
\item $\textbf{S}(1+2^{l-1},l)=2^{3n-l-2}$ if $3\leq l \leq n-1$.
\item $\textbf{S}(-1+2^{l-1},l)=2^{3n-3l+2}$ if $3\leq l \leq n-1$.
\item $\textbf{S}(1)= \textbf{S}(1,1)= 3\cdot 2^{3n-4}(n-1)$.
\item $\textbf{S}(2)=\textbf{S}(1,2)+\textbf{S}(-1,2)=3\cdot 2^{3n-5}(n-2)+3\cdot 2^{3n-6}$.
\item $\textbf{S}(l)= 3\cdot 2^{3n-l-3}(n-l)+3\cdot 2^{3n-3l}+2^{3n-l-2}+2^{3n-3l+2}$  if $3\leq l \leq n-2$. 
\item $\textbf{S}(n-1)= 3\cdot 2^3+2^5+2^{2n-1}$.
\item $\textbf{S}(n)=2^3+2^{2n-1}$.
\end{enumerate}
\par We conclude by computing:
\begin{eqnarray*}
|\Sppp| \textbf{S}_{2^n} &=& \textbf{S}(0) + \textbf{S}(1) +\textbf{S}(2) +\sum_{l=3}^{n-2}{\textbf{S}(l)} +\textbf{S}(n-1)+ \textbf{S}(n) 
\\&=& 3\cdot 2^{3n-2}(n-1) = |\Sppp|\times (n-1)
\end{eqnarray*}
\end{proof}

\subsection{Higher genus factors}

\vspace{4mm} The following theorem was shown in \cite{CNS} when $r$ is odd. We give a different argument and deal with the case $r=2$ by using the results on the genus one representations.
\begin{theorem}\label{irred_factors}
If $r$ is prime, the modules $U_r^{g, \pm }$ and $W_{r^n}^{g, \pm }$ are irreducible.
\end{theorem}
\vspace{2mm}
\par \begin{proof}
\par First let us handle the $U_r^{g, \pm }$ modules, when $r$ is prime. 
  Denote by  $\mathcal{A}$ the $\textbf{k}_r$-subalgebra of $\End(U_r)$ generated by the operators $\pi_r (\phi)$ for $\phi\in SL_2(\mathbb{Z})$ and by $\mathcal{B}$ the $\textbf{k}_r$-subalgebra of $\End(U_r^{\otimes g})$ generated by the operators $\pi_{r,g} (\phi)$ for $\phi \in Sp_{2g}( \mathbb{Z})$,  when $r$ is odd, and $\phi \in \widetilde{Sp_{2g}( \mathbb{Z})}$, when $r$ is even.
	\vspace{1mm}
	\par We  denote by $\mathcal{A'}$ and $\mathcal{B'}$ their commutant in $\End(U_r)$ and $\End(U_r^{\otimes g})$ respectively. We know from the genus one study that $\mathcal{A'}$ is generated by $ \mathds{1}$ and the symmetry $\theta\in \GL(U_r)$  sending $e_i$ to $e_{-i}$. There is a natural injection $i:\mathcal{A}\otimes \ldots\otimes \mathcal{A} \hookrightarrow \mathcal{B}$. Now using the fact that the commutant of a tensor product is the tensor product of the commutant we get: 
\begin{equation*}
\mathcal{B'}\subset i((\mathcal{A}\otimes\ldots\otimes \mathcal{A})')= i(\mathcal{A'}\otimes \ldots\otimes \mathcal{A'}) 
\end{equation*}
Note that when $r=2$ then $\theta=\mathds{1}$ so $\mathcal{B'}$ consists of scalar elements and $\pi_{2,g}=\pi_{2,g}^{+}$ is irreducible. We can thus suppose that $r$ is odd.
\vspace{2mm} \par
A generic element of $i(\mathcal{A'}\otimes \ldots\otimes \mathcal{A'}) $ has the form:
$$ C=\sum_{i\in I}{\lambda_i a_{i_1}\otimes \ldots \otimes a_{i_g}}, \mbox{ with } I\subset \{1,\ldots,p\}^g \mbox{ and } a_{i_k}= \mathds{1} \mbox{ or } \theta$$
To conclude we must show that $\mathcal{B'}$ is generated by $\mathds{1}\otimes \ldots \otimes \mathds{1}$ and $\theta\otimes \ldots \otimes \theta$, that is  show that if $C\in \mathcal{B'}$ then $a_{i_u}=a_{i_v}$ for all $i\in I$ and $u\neq v$.
\vspace{1mm}
\par Let us choose  $u,v$ and set $e:=e_1\otimes \ldots \otimes e_1$. We compute the commutator:
$$ [C,\pi_{r,g}(Z_{u,v})](e) = \sum_{i\in I} \lambda_i (A^{4\epsilon_i}-1)(a_{i_1}\otimes\ldots\otimes a_{i_g})(e) $$
where $\epsilon_i=0$ if $a_{i_u}=a_{i_v}$ and $\epsilon_i=1$ elsewhere. Since $A^4\neq 1$ and the family $\{ (a_{i_1}\otimes\ldots\otimes a_{i_g})(e), i\in I\}$ is free, the fact that $C$ is in the commutant of $\mathcal{B}$ implies that $\epsilon_i=0$ for all $i$ so  the two eigenspaces of $\theta\otimes \ldots \otimes \theta$ are irreducible.

\vspace{4mm}
\par Denote by  $\mathcal{C}$ the $\textbf{k}_{r^n}$-subalgebra of $\End(U_{r^n})$ generated by the operators $\pi_r (\phi)$ for $\phi\in SL_2(\mathbb{Z})$ and by  the $\textbf{k}_{r^n}$-subalgebra of $\End(U_{r^n}^{\otimes g})$ generated by the operators $\pi_{r,g} (\phi)$  for $\phi \in Sp_{2g}( \mathbb{Z})$,  when $r$ is odd, and $\phi \in \widetilde{Sp_{2g}( \mathbb{Z})}$, when $r$ is even.
	\vspace{1mm}
\par We  denote by $\mathcal{C'}$ and $\mathcal{D'}$ their commutant in $\End(U_{r^n})$ and $\End(U_{r^n}^{\otimes g})$ respectively. We know from the genus one study that $\mathcal{A'}$ is generated by $ \mathds{1}$ and  $\theta$. The natural injection $i:\mathcal{C}\otimes \ldots\otimes \mathcal{C} \hookrightarrow \mathcal{D}$ implies that: 
\begin{equation*}
\mathcal{D'}\subset i((\mathcal{C}\otimes\ldots\otimes \mathcal{C})')=  i(\mathcal{C'}\otimes \ldots\otimes \mathcal{C'}) 
\end{equation*}
\par Again we choose a generic element $C=\sum_{i\in I}{\lambda_i a_{i_1}\otimes \ldots \otimes a_{i_g}}\in i(\mathcal{C'}\otimes \ldots\otimes \mathcal{C'}) $ with $I\subset \{1,\ldots,{p^n}\}^g $ and $ a_{i_k}= \mathds{1} $ or $\theta$ and suppose that $C\in B'$. Now remember that $W_{r^n}$ is defined as the orthogonal of $\bar{U}_{r^{n-2}}=\Span (g_i)$ in $U_{r^{n}}$ and since $e_1$ is orthogonal to all $g_i$ we deduce that $e=e_1\otimes \ldots \otimes e_1\in W_{r^n}^{\otimes g}$. So the fact that the commutator $[C,\pi_{{r^{n+2}},g}(Z_{u,v})](e)$ is null if and only if $C$ is a linear combination of $\mathds{1}\otimes\ldots\otimes \mathds{1}$ and  $\theta\otimes \ldots \otimes \theta$ permits us to conclude.
\end{proof}
\vspace{3mm}
\par Finally the irreducibility of the factors coming from the decomposition at composite levels $p=r_1^{n_1}\ldots r_k^{n_k}$ follows, using the decomposition $(1)$ from Theorem $\ref{main_theorem}$, exactly as in the genus one case:
\begin{corollary}
All the modules of the form $B_{r_1}\otimes \ldots \otimes B_{r_k}$ with $r_1,\ldots,r_k$ distinct prime and $B_{r_i}=U_{r_i}^{g, \pm}$ or $W_{r_i^n}^{g, \pm}$, are irreducible and pairwise distinct.
\end{corollary}
\vspace{2mm}

\begin{proof}
Let $p=2^{\alpha}r_1^{n_1}\ldots r_k^{n_k}$ with $r_i$ some distinct odd primes. There is a group isomorphism between $\widetilde{Sp_{2g}(\mathbb{Z}/2p \mathbb{Z})}$ and $\widetilde{Sp_{2g}(\mathbb{Z}/2^{\alpha +1} \mathbb{Z})}\times Sp_{2g}(\mathbb{Z}/r_1^{n_1}\mathbb{Z})\times \ldots \times Sp_{2g}(\mathbb{Z}/r_k^{n_k}\mathbb{Z})$, if $p$ is even, and between $Sp_{2g}(\mathbb{Z}/p\mathbb{Z})$ and $Sp_{2g}(\mathbb{Z}/r_1^{n_1}\mathbb{Z})\times \ldots \times Sp_{2g}(\mathbb{Z}/r_k^{n_k}\mathbb{Z})$, if $p$ is odd. 
\vspace{2mm}
\par Denote by $\mathcal{A}_{p,g}$ the subalgebra of $\End(U_p^{\otimes g})$ generated by the operators $\pi_{p,g}(\phi)$. Using the first point of Theorem \ref{main_theorem}, we get an algebra isomorphism:
$$ \mathcal{A}_{p,g} \cong \mathcal{A}_{2^{\alpha},g} \otimes \mathcal{A}_{r_1^{n_1},g}\otimes \ldots \otimes \mathcal{A}_{r_k^{n_k},g} $$
\par We conclude using the fact that the commutant of a tensor product is the tensor product of the commutant and use Theorem \ref{irred_factors}.

\end{proof}

\vspace{5mm}

\section{Relation with the Witten-Reshetikhin-Turaev genus one representations}
\vspace{2mm}
\par We now give explicit isomorphisms between the submodules $U_p^-$ and the $SL_2(\mathbb{Z})$-modules $V_p$ defined in \cite{BHMV2} extending the relations in \cite{FK,LW} to the case where $p\equiv 3 \pmod{4}$. We include also their proof for self-completeness of the paper. Corollary \ref{coro} follows.
\par Denote by $S=\begin{pmatrix} 0&1\\ -1&0 \end{pmatrix}$ and $T=\begin{pmatrix} 1&-1 \\ 0 & 1 \end{pmatrix} $ the two generators of $SL_2(\mathbb{Z})$.
\vspace{2mm}
\par Using the basis $\{u_i, i\in I_p\}$ of $V_{p}$ defined in \cite{BHMV2}, where $$
I_p := \left\{ 
\begin{tabular}{ll}
$\{ 0, 1, 2, \ldots , r-2 \}$ & , if $ p=2r $ is even. \\
$\{ 0, 2, 4, \ldots , p-3 \}$ & , if $ p $ is odd. 
\end{tabular}
\right.$$ 
The Reshetikhin-Turaev representations in genus one are characterized by the projective class of the matrices:
$$ \begin{array}{cc}
\rho_p(T)=\left( A^{i(i+2)}\delta_{i,j} \right)_{i,j} & \rho_p(S)=c_p \left( (-1)^{i+j}[(i+1)(j+1)]\right)_{i,j} 
\end{array}$$
where we used the ring $\textbf{k'}_p:=\mathbb{Z} \left[ A,\frac{1}{p} \right] /(\phi_{2p}(A))$ , so $A$ is always a $2p$-th root of unity, and $c_p:=\frac{G(-1,0,2p)}{2p}$ when $p$ is even and $c_p:=\frac{G(-1,0,p)}{p}$ when $p$ is odd.
\vspace{5mm}
\par The following theorem was shown in \cite{FK} when $p$ is even and in \cite{LW} when $p\equiv 1 \pmod{4}$. We extend their proofs for  $p\equiv 3 \pmod{4}$. 
\vspace{2mm}
\begin{theorem}\label{link_bhmv}
For $p\geq 3$, the $SL_2(\mathbb{Z})$ projective modules $U_p^-$ and $V_{p}$ are projectively equivalent.
\end{theorem}
\vspace{4mm}
\par When $p$ is odd, the module $U_{p}$ is defined on the ring $\textbf{k}_p$, where $A$ is a primitive $p$-th root of unity, whereas $V_{p}$ is defined on $\textbf{k'}_p$, where $A$ is a primitive $2p$-th root of unity. In the preceding theorem, we turned $U_{p}^-$ into a $\textbf{k'}_p$-module using the ring morphism $\mu: \textbf{k'}_p \rightarrow \textbf{k}_p$ defined by $\mu(A)= A^4$.
\vspace{3mm}

\begin{proof}

\par When $p=2r$ is even, we define an isomorphism of $\textbf{k'}_{p}$-modules $\Psi : V_{p} \rightarrow U_{p}^{-}$ by $\Psi(u_i) = e_{r-i-1}-e_{r+i+1}$. We then compute the matrices of $\pi_p^-$ in the basis $(\Psi(u_i), i=0,1,\ldots r-2)$: 
\begin{eqnarray*}
\left<\Psi(u_j), \pi_p^- (T) \Psi(u_i) \right> &=& A^{(r-i-1)^2}\delta_{i,j}
\\ &=& A^{(r-1)^2} \cdot A^{-2ri} \cdot A^{i(i+2)}
\\ &=& A^{(r-1)^2} \rho_p(T)_{i,j}
\end{eqnarray*}
\vspace{1mm}
\begin{eqnarray*}
\left<\Psi(u_j), \pi_p^- (S) \Psi(u_i)\right> &=& c_p\left(A^{-2(r-i-1)(r-j-1)} - A^{2(r-i-1)(r-j-1)}\right) 
\\ &=& c_p \cdot A^{2r(i+j)}\left( A^{-2(i+1)(j+1)} - A^{2(i+1)(j+1)}\right)
\\ &=& - \rho_p(S)_{i,j}
\end{eqnarray*}
So $\pi_p^-$ and $\rho_p$ are projectively equivalent when $p$ is even.
\vspace{3mm}
\par Then when $p\geq 3$ is odd, we turn $U_{p}^-$ into a $\textbf{k'}_{p}$-module via the ring morphism $\mu : \textbf{k'}_{p}\rightarrow \textbf{k}_p$ defined by $\mu(A):=A^4$. We define an isomorphism $\Psi : V_{p}\rightarrow U_{p}^-$ of $\textbf{k'}_{p}$-modules via $\Psi(u_i):= e_{\frac{p-1-i}{2}} - e_{\frac{p+i+1}{2}}$ . We then compute the matrices of $\pi_p^-$ in the basis $(\Psi(u_i), i=0,2,4,\ldots p-3)$:
\begin{eqnarray*}
\left<\Psi(u_j), \pi_p^- (T) \Psi(u_i) \right> &=& \mu(A)^{\left( \frac{p-1-i}{2} \right)^2}\delta_{i,j}
\\ &=& A^{(p-i-1)^2} \delta_{i,j}
\\ &=& (-A)\cdot \rho_p(T)_{i,j}
\end{eqnarray*}
\vspace{1mm}
\begin{eqnarray*}
\left<\Psi(u_j), \pi_p^- (S) \Psi(u_i)\right> &=& c_p\left(\mu(A)^{-2(\frac{p-1-i}{2})(\frac{p-1-j}{2})}-\mu(A)^{2(\frac{p-1-i}{2})(\frac{p-1-j}{2})}\right) 
\\ &=& c_p \left(A^{-2(p-i-1)(p-1-j)}-A^{2(p-i-1)(p-1-j)}\right)
\\ &=& - \rho_p(S)_{i,j}
\end{eqnarray*}
And the proof is completed.
\end{proof}

\vspace{5mm}

\section{The Witten-Reshetikhin-Turaev TQFTs are determined by $3$-manifolds invariants without framed links}
\vspace{2mm}
\par In this section, we briefly review the universal construction of TQFTs of \cite{BHMV2} and prove Theorem \ref{th_BHMV}. For simplicity, we omit the complications due to the presence of an anomaly for it does not change the proof and refer to \cite{BHMV2, GM13} for more complete discussion. We also only write the proof when $p$ is even for the odd case easily follows using Theorem $1.5$ of \cite{BHMV2}.
\vspace{2mm}
\par Let $\mathcal{M}^{links}$ denotes the set of classes $(M,L)$ of closed oriented $3$ manifolds $M$ equipped with an embedded framed link $L\subset M$, modulo preserving-orientation homeomorphisms. In \cite{Li2, BHMV1}, the authors define a map $\tau_p : \mathcal{M}^{links}\rightarrow \mathbb{C}$ multiplicative for connected sums and sending the manifold $M$ with opposite orientation to the complex conjugate of the image of $M$.
\vspace{2mm}
\par Let $\Sigma$ be a closed oriented surface and $\mathcal{V} (\Sigma)$ be the complex vector space freely generated by  (homeomorphism classes of) elements $(M,\phi,L)$ where $M$ is a compact oriented three manifold, $\phi: \partial M\rightarrow \Sigma$ an orientation-preserving homeomorphism and $L\subset M$ is an embedded framed link (possibly empty). The space $\mathcal{V}(\Sigma)$ is naturally equipped with a bilinear form $\left<\cdot, \cdot\right>_p$ associated to $\tau_p$ defined as follows. If $\mathbb{M}_1=(M_1, \phi_1, L_1)$ and $\mathbb{M}_2=(M_2, \phi_2, L_2)$ are two cobordisms in $\mathcal{V}(\Sigma)$, we can glue them to obtain $\mathbb{M}_1\cup\mathbb{M}_2 :=(M_1 \cup_{\phi_1^{-1}\circ \phi_2}M_2, L_1\cup L_2) \in \mathcal{M}^{links}$. We then define $\left<\mathbb{M}_1, \mathbb{M}_2\right>_p := \tau_p (\mathbb{M}_1\cup \mathbb{M}_2)$ and extend the form to $\mathcal{V}(\Sigma)$ by bi-linearity. 
\vspace{2mm}
\par Eventually define the vector space:
$$ V_p(\Sigma) := \quotient{\mathcal{V}(\Sigma)}{\ker (\left<\cdot, \cdot\right>_p)}$$
\vspace{2mm}
\par By definition, any cobordism $\mathbb{M}\in \mathcal{V}(\Sigma)$ defines a vector $Z_p(\mathbb{M})\in V_p(\Sigma)$ by passing to the quotient. Moreover if $\mathbb{M}$ is a cobordism between to surfaces $\Sigma_1$ and $\Sigma_2$, we can associate a linear map $V_p(\mathbb{M}) : V_p(\Sigma_1) \rightarrow V_p(\Sigma_2)$ by sending $Z_p(\mathbb{M}')$ to $Z_p(\mathbb{M}\circ\mathbb{M}')$. Such a functorial assignation $\Sigma\rightarrow V_p(\Sigma)$ and $\mathbb{M}\rightarrow V_p(\mathbb{M})$ is what is called a TQFT. Note that the spaces $U_{p,g}$ of the \weil representations also fit into this framework (see \cite{GU, Koju_thesis}).
\vspace{2mm}
\par Denote by $X_p(\Sigma)\subset V_p(\Sigma)$ the subspace generated by classes of cobordisms with an empty link. Theorem \ref{th_BHMV} states that whenever $4$ does not divide $p$, then $X_p(\Sigma)=V_p(\Sigma)$.

By construction the subspace $X_p(\Sigma)$  is determined by the restriction of the three manifolds invariant $\tau_p$ to the subset  $\mathcal{M}\subset\mathcal{M}^{links}$ of closed oriented three manifolds without framed links.
\vspace{3mm}
\par We now turn to the proof of Theorem \ref{th_BHMV}. Simply denote by $V_p$ the space $V_p(S^1\times S^1)$ as in the previous section. Let $u_1:=Z_p(D^2\times S^1, \id, L) \in V_p$ be the vector associated to the manifold $D^2\times S^1$ with trivial boundary identification and the link $L=\{ 0\}\times S^1\subset D^2\times S^1$ with parallel framing. Theorem \ref{th_BHMV} easily follows from the following:
\vspace{2mm}
\begin{lemma}\label{lemma_un} If $4$ does not divide $p$, then $u_1\in X_p(S^1\times S^1)$.
\end{lemma}
\vspace{2mm}
\begin{proof}[Proof of Theorem \ref{th_BHMV} using Lemma \ref{lemma_un}]
\par The following argument is the same as Robert's argument in \cite{Ro} who proved Theorem \ref{th_BHMV} when $p$ is prime. We briefly reproduce it for self-completeness of the paper. Let $\mathbb{M}=(M,\phi, L)\in\mathcal{V}(\Sigma)$ be a cobordism and $Z_p(\mathbb{M})\in V_p(\Sigma)$ its class in the quotient. We have to show that $Z_p(\mathbb{M})$ is a linear combination of vectors associated to cobordisms without links, so we suppose that $L$ is not empty. 
\vspace{2mm}
\par Let $L_i\subset L$ be a connected component and choose $N_i$ a tubular neighborhood of $L_i$ in $M$ homeomorphic to $D^2\times S^1$. Writing $\mathbb{M}\setminus N_i= (M\setminus N_i, \phi\cup \phi_{N_i}, L\setminus L_i)\in \mathcal{V}(\Sigma \bigsqcup S^1\times S^1)$, we have $(\mathbb{M}\setminus N_i)\cup_{\partial N_i} (N_i, \phi_{N_i}, L_i)= \mathbb{M}$.
 \vspace{2mm} \par Passing to the quotient, we get $Z_p(\mathbb{M})=V_p(\mathbb{M}\setminus N_i)\circ Z_p(N_i, \phi_{N_i}, L_i)$, where $V_p(\mathbb{M}\setminus N_i)$ is a linear map from $V_p$ to $V_p(\Sigma)$ and $Z_p(N_i, \phi_{N_i}, L_i)$ is the vector $u_1\in V_p$. Lemma \ref{lemma_un} implies the existence of three manifolds $\mathbb{M}_1, \ldots, \mathbb{M}_k$ bounding $S^1\times S^1$ without framed links embedded and scalars $\lambda_1, \ldots, \lambda_k$ in $\mathbb{C}$ such that $u_1= \sum_i \lambda_i Z_p(M_i)$. It follows that:
 $$ Z_p(\mathbb{M})=\sum_i \lambda_i Z_p((\mathbb{M}\setminus N_i) \circ \mathbb{M}_i))$$
 \par Thus $Z_p(\mathbb{M})$ is a linear combination of vectors associated to cobordisms with one component less than $L$. We conclude by induction on the number of components of $L$.
 \end{proof}
\vspace{3mm}
\par The proof of Lemma \ref{lemma_un} relies on the fact that $X_p\subset V_p$ is invariant under the action of  $SL_2(\mathbb{Z})$ on $V_p$. Let $u_0 \in V_p$ denotes the vector associated to $D^2\times S^1$ without framed links embedded and let $\Lambda_0, \Lambda_1 \subset V_p$ be the $SL_2(\mathbb{Z})$ cyclic subspaces associated to $u_0$ and $u_1$ respectively.
\begin{lemma}\label{lemma_deux} If $4$ does not divide $p$, then $\Lambda_0 = \Lambda_1$. \end{lemma}

\begin{proof} Since $\Lambda_0$ and $\Lambda_1$ are $SL_2(\mathbb{Z})$ invariant subspaces by definition, we have to show that for any irreducible subspace $B\subset V_p\cong U_p^-$, we have $\Lambda_0\cap B = \Lambda_1\cap B$. Note that $\Lambda_i\cap B$ is either $\{0\}$ or $B$.
\vspace{2mm}
\par Using the identification $\Psi: V_p \cong U_p^-$ of (the proof of) Theorem \ref{link_bhmv} and Corollary \ref{coro}, we know explicit basis for such irreducible modules. Denote by $\Lambda_i ':= \Psi(\Lambda_i)\subset U_p^-$ and remark that if $p=2r$, we have:
$$ \begin{array}{ll} \psi(u_0)= e_{r-1}-e_{r+1} & \psi(u_1)=e_{r-2}-e_{r+2} \end{array} $$
\vspace{2mm}
\par Note $p=2r_1^{n_1}\ldots r_k^{n_k}$ the decomposition of $p$ in primes numbers, and choose $B=U_2\otimes B_1\otimes \ldots \otimes B_k \subset U_p^-$ an irreducible submodule as in Corollary \ref{coro}. We have to study whether the projection of $\psi(u_i)$ on $B$ is null or not.
\vspace{3mm}
\par First consider the case where there exists in index $i$ such that $n_i\geq 2$ and $B_i \neq W_{r_i^{n_i}}^{\pm}$. Then $B_i\subset U_{r_i^{n_i-2}}^{\pm}$ which is included in the subspace spanned by vectors $e_k$ such that $r_i$ divides $k$. But clearly $r_i$ does not divide $r-1, r+1, r-2$ nor $r+2$ thus the projection of both $\psi(u_0)$ and $\psi(u_1)$ on $B$ is null and we have $\Lambda_0'\cap B = \Lambda_1'\cap B = \{0\}$.
\vspace{2mm}
\par Next suppose that for each $i$ such that $n_i\geq 2$, we have $B_i= W_{r_i^{n_i}}^{\epsilon_i}$ where $\epsilon_i$ is either $-1$ or $+1$. Given two integers $x$ and $n$, we will denote by $[x]_n\in \mathbb{Z}/n\mathbb{Z}$ the class of $x$ modulo $n$. Let $x$ be any integer such that none of the $r_i$ divides $x$. Set:
$$ v_B:= e_{[x]_2}\otimes e_{[x]_{r_1^{n_1}}}^{\epsilon_1}\otimes \ldots \otimes e_{[x]_{r_k^{n_k}}}^{\epsilon_k}\in B$$
where we used the notation $e_i^\pm := e_i \pm e_{-i}$. By using the fact that $<e_i, e_i^{\epsilon}>=1$ and $<e_{-i}, e_i^{\epsilon}>= (-1)^{\frac{1-\epsilon}{2}}$, we compute: 
\begin{eqnarray*} <v_B, e_x-e_{-x}> &=& \left< e_{[x]_2}\otimes e_{[x]_{r_1^{n_1}}}^{\epsilon_1}\otimes \ldots \otimes e_{[x]_{r_k^{n_k}}}^{\epsilon_k} , e_{[x]_2 \otimes e_{[x]_{r_1^{n_1}}}}\otimes \ldots \otimes e_{[x]_{r_k^{n_k}}} \right> \\
& &-  \left< e_{[x]_2}\otimes e_{[x]_{r_1^{n_1}}}^{\epsilon_1}\otimes \ldots \otimes e_{[x]_{r_k^{n_k}}}^{\epsilon_k} , e_{[-x]_2} \otimes e_{[-x]_{r_1^{n_1}}}\otimes \ldots \otimes e_{[-x]_{r_k^{n_k}}} \right>\\
&=& 1-(-1)^{\sum_i \frac{1-\epsilon_i}{2}} =2\neq 0
\end{eqnarray*}

\par where we used in the last line the fact that there is an odd number of $i$ such that $\epsilon_i=-1$ for $B\subset U_p^-$. In particular the orthogonal projection of $e_x^-$ on $B$ is non-trivial whenever none of the $r_i$ divides $x$. Applying this to $x=r-1$ and $x=r-2$, we get that $\Lambda_0'\cap B=\Lambda_1'\cap B=B$.
\end{proof}
\vspace{2mm}
\begin{proof}[Proof of Lemma \ref{lemma_un}] Since $u_0$ belongs to $X_p$ by definition and that $X_p$ is invariant under the action of $SL_2(\mathbb{Z})$, we have $\Lambda_0\subset X_p$. Now Lemma \ref{lemma_deux} implies that $\Lambda_1\subset X_p$ thus $u_1\in X_p$.
\end{proof}

\vspace{5mm}


\bibliographystyle{plain}
\bibliography{biblio}

\end{document}

%% file: macrosmath.tex
\newcommand{\weil}{Weil }

\newcommand{\id}{id}
\newcommand{\zn}{\mathbb{Z}/p\mathbb{Z}}
\newcommand{\znn}{\mathbb{Z}/2p\mathbb{Z}}
\newcommand{\Sp}{SL_2(\mathbb{Z}/p\mathbb{Z})}
\newcommand{\Spp}{SL_2(\mathbb{Z}/2p\mathbb{Z})}

\newcommand{\Sppp}{SL_2(\mathbb{Z}/2^n\mathbb{Z})}
\newcommand{\Span}{\operatorname{Span}}
\newcommand{\Stab}{\operatorname{Stab}}
\newcommand{\znnn}{\mathbb{Z}/2^n\mathbb{Z}}
\newcommand{\ssp}{Sp_{2g}(\mathbb{Z})}
\newcommand{\sph}{\widetilde{Sp_{2g}(\mathbb{Z})}}
\newcommand{\spp}{Sp_{2g}(\mathbb{Z}/p\mathbb{Z})}
\newcommand{\sppp}{Sp_{2g}(\mathbb{Z}/2p\mathbb{Z})}

\newcommand{\GL}{\operatorname{GL}}
\newcommand{\PGL}{\operatorname{PGL}}
\newcommand{\End}{\operatorname{End}}

\newcommand{\tr}{\operatorname{Tr}}
\newcommand{\Mod}{\operatorname{Mod}}
\newcommand{\Add}{\operatorname{Add}}

\newcommand{\quotient}[2]{{\raisebox{.2em}{$#1$}\left/\raisebox{-.2em}{$#2$}\right.}}

%% file: Korinman_WeilRepresentations_ArXiv.bbl
\def\cprime{$'$}
\begin{thebibliography}{10}

\bibitem{AR}
A.~Adler and S.~Ramanan.
\newblock {\em Moduli of abelian varieties}, volume 1644 of {\em Lecture Notes
  in Mathematics}.
\newblock Springer-Verlag, Berlin, 1996.

\bibitem{BE}
B.C. Berndt and R.J. Evans.
\newblock The determination of {G}auss sums.
\newblock {\em Bull. Amer. Math. Soc. (N.S.)}, 5(2):107--129, 1981.

\bibitem{HB}
M.~V. Berry and J.~H. Hannay.
\newblock Quantization of linear maps on a torus-{F}resnel diffraction by a
  periodic grating.
\newblock {\em Phys. D}, 1(3):267--290, 1980.

\bibitem{BHMV1}
C.~Blanchet, N.~Habegger, G.~Masbaum, and P.~Vogel.
\newblock Three-manifold invariants derived from the {K}auffman bracket.
\newblock {\em Topology}, 31(4):685--699, 1992.

\bibitem{BHMV2}
C.~Blanchet, N.~Habegger, G.~Masbaum, and P.~Vogel.
\newblock Topological quantum field theories derived from the {K}auffman
  bracket.
\newblock {\em Topology}, 34(4):883--927, 1995.

\bibitem{BDB}
A.~Bouzouina and S.~De~Bi\`evre.
\newblock Equipartition of the eigenfunctions of quantized ergodic maps on the
  torus.
\newblock {\em Comm. Math. Phys.}, 178:83--105, 1996.

\bibitem{Che}
C.~Chevalley.
\newblock {\em Introduction to the {T}heory of {A}lgebraic {F}unctions of {O}ne
  {V}ariable}.
\newblock Mathematical Surveys, No. VI. American Mathematical Society, New
  York, N. Y., 1951.

\bibitem{CNS}
G.~Cliff, D.~McNeilly, and F.~Szechtman.
\newblock Weil representations of symplectic groups over rings.
\newblock {\em J. London Math. Soc. (2)}, 62(2):423--436, 2000.

\bibitem{DG}
F.~Deloup and C.~Gille.
\newblock Abelian quantum invariants indeed classify linking pairings.
\newblock {\em J. Knot Theory Ramifications}, 10(2):295--302, 2001.
\newblock Knots in Hellas '98, Vol. 2 (Delphi).

\bibitem{Dem}
M.~Demazure.
\newblock {\em Cours d'alg{\`e}bre}.
\newblock Cassini, 1997.

\bibitem{FNdB}
F.~Faure, S.~Nonnenmacher, and S.~De~Bi{\`e}vre.
\newblock Scarred eigenstates for quantum cat maps of minimal periods.
\newblock {\em Comm. Math. Phys.}, 239(3):449--492, 2003.

\bibitem{FK}
M.~Freedman and V.~Krushkal.
\newblock On the asymptotics of quantum {${\rm SU}(2)$} representations of
  mapping class groups.
\newblock {\em Forum Math.}, 18(2):293--304, 2006.

\bibitem{Fu2}
L.~Funar.
\newblock Some abelian invariants of 3-manifolds.
\newblock {\em Rev. Roumaine Math. Pures Appl.}, 45(5):825--861 (2001), 2000.

\bibitem{FP}
L.~{Funar} and W.~{Pitsch}.
\newblock {Finite quotients of symplectic groups vs mapping class groups},
  2011.

\bibitem{GU}
R.~Gelca and A.~Uribe.
\newblock From classical theta functions to topological quantum field theory,
  2010.

\bibitem{GM13}
P.M. Gilmer and G.~Masbaum.
\newblock Maslov index, lagrangians, mapping class groups and {TQFT}.
\newblock {\em Forum Math.}, 25(5):1067--1106, 2013.

\bibitem{Go}
T.~Gocho.
\newblock The topological invariant of three-manifolds based on the {${\rm
  U}(1)$} gauge theory.
\newblock {\em J. Fac. Sci. Univ. Tokyo Sect. IA Math.}, 39(1):169--184, 1992.

\bibitem{Ig}
J.~Igusa.
\newblock On the graded ring of theta-constants. {II}.
\newblock {\em Amer. J. Math.}, 88:221--236, 1966.

\bibitem{Kl}
H.~D. Kloosterman.
\newblock The behaviour of general theta functions under the modular group and
  the characters of binary modular congruence groups. {I}.
\newblock {\em Ann. of Math. (2)}, 47:317--375, 1946.

\bibitem{Koju_thesis}
J.~{Korinman}.
\newblock {\em {On some quantum representations of the mapping class groups of
  surfaces}}.
\newblock PhD thesis, Institut Fourier, 2014.

\bibitem{KR}
P.~Kurlberg and Z.~Rudnick.
\newblock Hecke theory and equidistribution for the quantization of linear maps
  of the torus.
\newblock {\em Duke Math. J.}, 103(1):47--77, 2000.

\bibitem{LW}
M.~Larsen and Z.~Wang.
\newblock Density of the {SO}(3) {TQFT} representation of mapping class groups.
\newblock {\em Comm. Math. Phys.}, 260(3):641--658, 2005.

\bibitem{Li1}
W.~B.~R. Lickorish.
\newblock A finite set of generators for the homeotopy group of a
  {$2$}-manifold.
\newblock {\em Proc. Cambridge Philos. Soc.}, 60:769--778, 1964.

\bibitem{Li2}
W.~B.~R. Lickorish.
\newblock Invariants for {$3$}-manifolds from the combinatorics of the {J}ones
  polynomial.
\newblock {\em Pacific J. Math.}, 149(2):337--347, 1991.

\bibitem{LV}
G.~Lion and M.~Vergne.
\newblock {\em The {W}eil representation, {M}aslov index and theta series},
  volume~6 of {\em Progress in Mathematics}.
\newblock Birkh\"auser Boston, Mass., 1980.

\bibitem{MOO}
H.~Murakami, T.~Ohtsuki, and M.~Okada.
\newblock Invariants of three-manifolds derived from linking matrices of framed
  links.
\newblock {\em Osaka J. Math.}, 29(3):545--572, 1992.

\bibitem{Prasad}
A.~Prasad.
\newblock On character values and decomposition of the {W}eil representation
  associated to a finite abelian group.
\newblock {\em J. Anal.}, 17:73--85, 2009.

\bibitem{RT}
N.~Reshetikhin and V.~G. Turaev.
\newblock Invariants of {$3$}-manifolds via link polynomials and quantum
  groups.
\newblock {\em Invent. Math.}, 103(3):547--597, 1991.

\bibitem{Roberts94}
J.~Roberts.
\newblock Skeins and mapping class groups.
\newblock {\em IMath. Proc. Cambridge Phil.Soc.}, 115:53--77, 1994.

\bibitem{Ro}
J.~Roberts.
\newblock Irreducibility of some quantum representations of mapping class
  groups.
\newblock {\em J. Knot Theory Ramifications}, 10(5):763--767, 2001.
\newblock Knots in Hellas '98, Vol. 3 (Delphi).

\bibitem{Se}
I.E. Segal.
\newblock Lectures at the $1960$ boulder summer seminar, 1962.

\bibitem{Serre}
J-P. Serre.
\newblock {\em Linear representations of finite groups}, volume~42 of {\em
  Graduate Texts in Mathematics}.
\newblock Springer Verlag, 1977.

\bibitem{Sh}
D.~Shale.
\newblock Linear symmetries of free boson fields.
\newblock {\em Trans. Amer. Math. Soc.}, 103:149--167, 1962.

\bibitem{Shi}
G.~Shimura.
\newblock Moduli and fibre systems of abelian varieties.
\newblock {\em Ann. of Math. (2)}, 83:294--338, 1966.

\bibitem{Weil}
A.~Weil.
\newblock Sur certains groupes d'op\'erateurs unitaires.
\newblock {\em Acta Math.}, 111:143--211, 1964.

\end{thebibliography}
